\RequirePackage{fix-cm}
\documentclass[smallextended]{svjour3}       
\smartqed  
\usepackage{graphicx}
\usepackage{amsmath,amssymb,amsfonts,latexsym}
\usepackage{color}
\usepackage{eucal}
\usepackage{xspace}
\usepackage{hyperref}

\usepackage{bm}
\usepackage[toc,page]{appendix}

\usepackage[round]{natbib}

\begin{document}
	
	\title{Statistics of topological RNA  structures 
	}
	
	
	\author{Thomas J. X. Li         \and
		Christian M. Reidys 
	}
	

	\institute{Thomas J. X. Li \at
		Biocomplexity Institute of Virginia Tech\\
		Blacksburg, VA 24061, USA\\
		\email{thomasli@vbi.vt.edu}           
		\and
		Christian M. Reidys \at
		Biocomplexity Institute of Virginia Tech\\
		Blacksburg, VA 24061, USA\\
		\email{duckcr@vbi.vt.edu}
	}
	
	\date{Received: date / Accepted: date}

	\maketitle
	
	\begin{abstract}
		 In this paper we study properties of topological RNA structures, i.e.~RNA contact structures
		 with cross-serial interactions that are filtered by their topological genus. RNA secondary
		 structures within this framework are topological structures having genus zero.
		 We derive a new bivariate generating function whose singular expansion allows us to analyze
		 the distributions of arcs, stacks, hairpin- , interior- and multi-loops.
                 We then extend this analysis to H-type pseudoknots, kissing hairpins as well as $3$-knots
                 and compute their respective expectation values.
		 Finally we discuss our results and put them into context with data obtained by uniform sampling
		 structures of fixed genus.

	\end{abstract}


\section{Introduction}

An RNA sequence is described by its primary structure, a linear oriented sequence of the nucleotides and can be viewed
as a string over the alphabet $\{\mathbf{A},\mathbf{U},\mathbf{G},\mathbf{C}\}$. An RNA strand folds by forming hydrogen
bonds between pairs of nucleotides according to Watson-Crick \textbf{A-U}, \textbf{C-G} and wobble \textbf{U-G} base-pairing
rules. The secondary structure encodes this bonding information of the nucleotides irrespective of the actual spacial embedding.
More than three decades ago, Waterman and colleagues pioneered the combinatorics and prediction of RNA secondary
structures~\citep{Waterman:78s,Waterman:79a,Waterman:78aa,Howell:80,Waterman:94a,Waterman:93}. Represented as a \emph{diagram}
by drawing its sequence on a horizontal line and each base pair as an arc in the upper half-plane, RNA secondary
structure contains no crossing arcs (two arcs $(i_1,j_1)$ and $(i_2,j_2)$ cross if the nucleotides appear in the order
$i_1<i_2<j_1<j_2$ in the primary structure).

In fact, it is well-known that there exist cross-serial interactions, called pseudoknots in RNA~\citep{Westhof:92}, see Fig.~\ref{F:pk1}.
RNA structures with cross-serial interactions are of biological significance, occur often in practice and are found to be functionally
important in tRNAs, RNAseP~\citep{Loria:96}, telomerase RNAs~\citep{Staple:05,Chen:00}, and ribosomal RNAs~\citep{Konings:95}.
Cross-serial interactions also appear in  plant viral RNAs and
\textit{in vitro}  RNA evolution experiments have produced pseudoknotted RNA families, when binding HIV-1 reverse
transcriptase~\citep{Tuerk:92}.

However, little is known w.r.t.~their basic statistical properties. In order to be able to identify biological features this paper establishes the ``base line'' by studying random RNA pseudoknot structures.
Basic questions here are for
instance how the statistics of hairpin-, interior- and multi-loops change in the context of cross-serial bonds and to study
new features as H-loops and kissing hairpins.

The key to organize and filter structures with cross-serial interactions is to introduce topology. The idea is simple: instead
of drawing a structure in the plane (sphere) we draw it on more sophisticated orientable surfaces. The advantage of this is that
this presentation allows to eliminate any cross-serial interactions. RNA secondary structure fits seamlessly into
this framework, since these are exactly topological structures of genus zero, i.e.~structures that can be drawn on a sphere
without crossings.

The topology of RNA structures has first been studied in~\citet{Waterman:93,Penner:03} and the classification and expansion
of RNA structures including pseudoknots in terms of the topological genus of an associated fatgraph via matrix theory
in~\citet{Orland:02,Vernizzi:05,Bon:08}. The computation of the genus of a fatgraph dates back to~\citet{Euler:52} and was
applied to RNA structures by~\citet{Orland:02,Bon:08}.~\citet{reidys:2013} study topological RNA structures of higher genus
and associate them with Riemann's Moduli space in~\citet{Penner:03}. In~\citet{Huang:11}, a loop-based
folding algorithm of topological RNA structures is given.~\citet{Huang:13} present a linear time uniform sampling for these topological
structures. Recently, ~\citet{Huang:16} introduce a stochastic context-free grammar (SCFG) facilitating the efficient Boltzmann-sampling
of RNA pseudoknotted structures.

An RNA pseudoknotted structure is modeled by augmenting the notion of a graph, as an orientable fatgraph. This is obtained by replacing
vertices by discs and edges by ribbons in the diagram representation.
Gluing the sides of these ribbons, creates a closed, orientable surface of genus $g$, which is a connected sum of $g$ tori
(see Section~\ref{S:Back}).
As topological genus completely characterizes any closed, orientable surface~\citep{Massey:69},
RNA structures are filtered by just one parameter, the genus of their associated
fatgraphs and RNA secondary structures are exactly structures of genus zero. The genus of a structure is not affected by removing noncrossing
arcs or collapsing a stack into a singleton. This leads to the notion of a \emph{shape}, a diagram which contains no unpaired vertices and no arcs of length $1$, and
in which any stack has length one.
The generating function of shapes of genus $g$ is computed in~\cite{Huang:14} (see also~\citet{Li:14}) and deeply rooted in the work
of~\citet{Harer:86}. Another notion called \emph{irreducible shadow} has been studied in~\cite{Han:14,Li:13}. Intuitively, an irreducible
shadow can be viewed as the minimal building block of a shape. A similar notion is
discussed in~\citet{Orland:02,Vernizzi:05,Bon:08}. A shape can be constructed by nesting and concatenating irreducible shadows.
The generating function of irreducible shadows is computed in~\cite{Han:14}.
\begin{figure}
	\centering
	\includegraphics[width=1\textwidth]{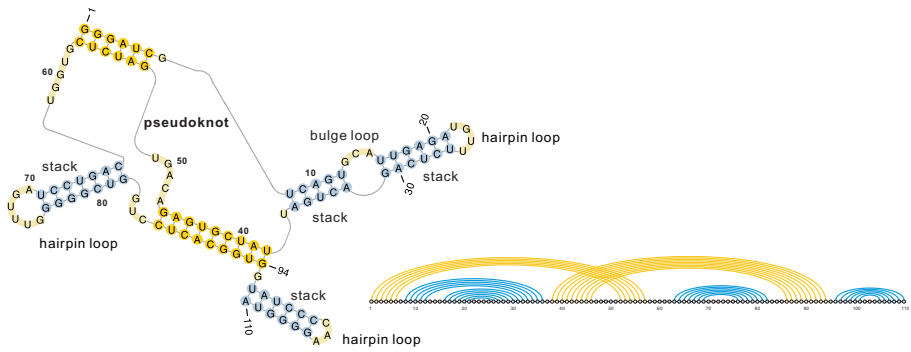}
	\caption
	    {\small The secondary structure and diagram representation of ribox02, a ribozyme that catalyzes an oxidation-reduction
              reaction~\citep{Tsukiji:03}.   Stacks and loops are indicated by different colors.
	}
	\label{F:pk1}
\end{figure}

This paper is motivated by the observation that uniformly sampled topological RNA structures exhibit sharply concentrated number of
arcs and that the distribution is unaffected by their genus.
To analyze this, we introduce a novel bivariate generating function for RNA structures of genus $g$ filtered by the number of arcs.
For $g\geq 1$, we compute these by employing the shape polynomial~\citep{Huang:14}.
We show that the singular expansion exhibits an exponential factor independent of $g$ and a subexponential factor having degree
$\frac{6g-3}{2}$, closely related to the degree of the shape polynomial. We shall prove a central limit theorem for the distribution
of the number of arcs in structures of genus $g$. We then extend this analysis to stacks, hairpin loops, bulges, interior loops and
multi-loops, generalizing the results of~\citet{reidys:2013} and results for secondary structures of~\citet{Hofacker:98}
and~\citet{Barrett:16}.

We furthermore establish the block decomposition for RNA pseudoknot structures, see Fig.~\ref{F:pk1},  generalizing the standard
decomposition of secondary structures of ~\cite{Waterman:78s}.  
Augmenting the bivariate generating polynomial of shapes by marking specific types of irreducible shadows, we show that the
expectation value of H-type, kissing hairpin, $3$-knot and $4$-knot pseudoknots, see Fig.~\ref{F:g1pk}, in uniformly generated structures
of any genus is $ O\left(n^{-1}\right)$, $ O(n^{-\frac{1}{2}})$, $ O(n^{-\frac{1}{2}})$ and $ O(1)$, respectively, see Fig.~\ref{F:p-Loop2}.

\begin{figure}
	\centering
	\includegraphics[width=\textwidth]{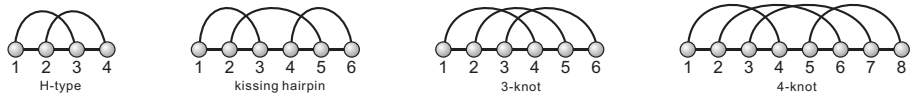}
	\caption
	{\small Four types of pseudoknots of genus $1$.
	}
	\label{F:g1pk}
\end{figure}

\begin{figure}
	\centering
	\includegraphics[width=.7\textwidth]{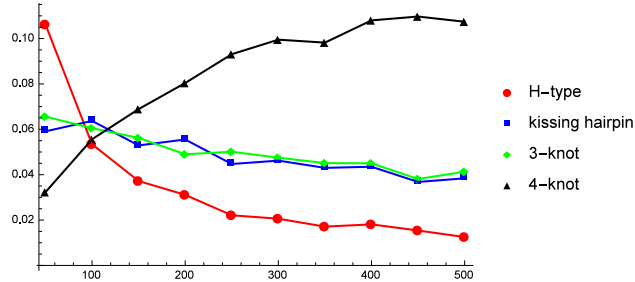}
	\caption
	    {\small The expectation of different types of pseudoknots in $10^5$ uniformly generated RNA structures of genus
              $2$ as a function of sequence length.
	}
	\label{F:p-Loop2}
\end{figure}


This paper is organized as follows: In Section~\ref{S:Back}, we provide some basic facts of the fatgraph model, linear chord diagram
and RNA secondary structures. We compute the bivariate generating function of structures of length $n$ having $l$ arcs  and fixed
genus and its
asymptotic expansion in Section~\ref{S:GF}. In Section~\ref{S:PT}, we  prove a central limit theorem for the distribution of the
number of arcs. In Section~\ref{S:other}, we extend these results to all other loop-types.
We present the block decomposition for pseudoknot structures in Section~\ref{S:loop}
and compute the expectation values of various types of irreducible shadows in Section~\ref{S:ploop}. In Section~\ref{S:pf} we present
all proofs. We conclude with Section~\ref{S:Dis}, where we integrate and discuss our findings.


\section{Basic facts}\label{S:Back}




A \emph{diagram} (or partial linear chord diagram in~\citet{reidys:2013})  is a labeled graph over the vertex set $\{1, \dots, n\}$  whose vertices are arranged in a horizontal line and arcs are drawn in the upper half-plane. Clearly, vertices and arcs correspond to nucleotides and base pairs, respectively. The number of nucleotides
is called the length of the structure. The length of an arc $(i,j)$ is defined as $j-i$ and an arc of length $k$ is called a $k$-arc.
The backbone of a diagram is the sequence of consecutive integers $(1,\dots,n)$ together with the edges $\{\{i,i+1\}\mid 1\le i\le n-1\}$.
We shall distinguish the backbone edge $\{i,i+1\}$ from the arc $(i,i+1)$, which we refer to as a \emph{$1$-arc}.
Two arcs $(i_1,j_1)$ and $(i_2,j_2)$ are \emph{crossing} if $i_1<i_2<j_1<j_2$.  An
RNA \emph{secondary structure} is defined as a diagram without $1$-arcs and crossing arcs~\citep{Waterman:78s}.
Pairs of nucleotides may form Watson-Crick \textbf{A-U}, \textbf{C-G} and wobble \textbf{U-G} bonds labeling the above mentioned
arcs.

To extract topological properties of the cross-serial interactions in pseudoknot structures, we need to  enrich diagrams to fatgraphs.
By gluing the sides of ribbons these induce orientable surfaces, whose topological genera give rise to a filtration.
Combinatorially, a \emph{fatgraph} is a graph  together with a collection of cyclic orderings on the half-edges incident to each vertex
and is usually obtained by expanding each vertex to a disk and fattening the edges into (untwisted) ribbons or bands  such that the
ribbons connect the disks ingiven cyclic orderings. The specific drawing of a diagram $G$ with its arcs in the upper half-plane determines
a collection of cyclic orderings on the half-edges of the underlying graph incident on each vertex,
thus defining a corresponding fatgraph $\mathbb{G}$, see Fig.~\ref{F:fat}.
Accordingly, each fatgraph $\mathbb{G}$ determines an associated orientable surface
$F(\mathbb{G})$ with boundary \citep{Loebl:08,penner:2010}, which
contains $G$ as a deformation retract \citep{Massey:69}, see Fig.~\ref{F:fat}.
Fatgraphs were first applied to RNA secondary structures in
\cite{Waterman:93} and \cite{Penner:03}.

\begin{figure}
	\centering
	\includegraphics[width=1\textwidth]{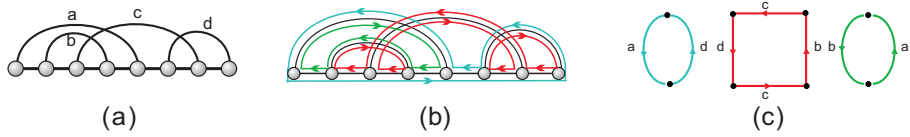}
	\caption
	{\small A diagram, its corresponding fatgraph and boundary components. (a) A diagram $G$ of genus $1$. (b) Its corresponding fatgraph $\mathbb{G}$ represented by an orientable surface
		$F(\mathbb{G})$ with three boundary components. (c) Three boundary components viewed as polygons.
	}
	\label{F:fat}
\end{figure}
The surface $F(\mathbb{G})$ is characterized up to homeomorphism by its genus $g \geq 0$ and the number $r \geq 1$ of boundary components,
which are associated to $G$ itself. Filling the boundary components with polygons we can
pass from $F(\mathbb{G})$ to a surface without boundary.
Euler characteristic, $\chi$, and genus, $g$, of this surface are connected via
$$
\chi =  v - e + r\quad\text{\rm and }\quad
g  =  1-\frac{1}{2}\chi,
$$
where $v,e,r$ denote the number of disks, ribbons and boundary components in
$\mathbb{G}$, \cite{Massey:69}. Any topological RNA structure having genus greater than or equal to one is referred to as RNA pseudoknot structure (pk-structure). We shall use the term topological structure when we wish to emphasize their filtration by topological genus.

In this paper, we consider RNA structures subject to two types of restrictions
\begin{itemize}
	\item \emph{Minimum arc-length} restrictions, arising from the rigidity of the backbone. RNA secondary structures having minimum arc-length two were studied by Waterman \citep{Waterman:78s}. Arguably, the most realistic cases is $\lambda=4$~\citep{Stein:79}, and
	RNA folding algorithms, generating minimum free energy structures, implicitly satisfy this
	constraint\footnote{each hairpin loop contains at least three unpaired bases} for energetic reasons,
	\item \emph{Minimum stack-length} restrictions.  A \emph{stack} of length $r$ is a maximal sequence of "parallel" arcs, $((i,j),(i+1,j-1),\ldots,(i+(r-1),j-(r-1)))$.
	Stacks of length $1$
	are energetically unstable and we find typically stacks of length at least two or three in biological structures~\citep{Waterman:78s}.
	A structure, $S$, is \emph{$r$-canonical} if any of its stacks has length at least three.
\end{itemize}

Let $d_{g,\lambda}^{[r]} (n)$  denote the
number of $r$-canonical topological RNA structures
of $n$ nucleotides and genus $g$, with minimum arc-length $\lambda$. Let furthermore
$d_{g,\lambda}^{[r]}(n,l)$  denote the number of $r$-canonical topological RNA structures of genus $g$  filtered by the number of arcs. Let $\mathbf{D}_{g,\lambda}^{[r]} (x,y)=\sum_{n,l} d_{g,\lambda}^{[r]}(n,l)x^n y^l$  denote the corresponding bivariate generating function.

We shall write $d_{g} (n,l)$ and $\mathbf{D}_{g}(x,y)$ instead of $d_{g,\lambda}^{[r]} (n,l)$ and $\mathbf{D}_{g,\lambda}^{[r]} (x,y)$. 


A \emph{linear chord diagram} is a diagram without unpaired vertices.
Three decades ago, \citet{Harer:86} discovered the celebrated two-term recursion for the number $c_g(n)$ of linear chord diagrams of genus $g$ with $n$ arcs.
\begin{theorem}[\citet{Harer:86}]
	The number 	$c_g(n)$ satisfy the recursion
	\begin{equation}\label{E:bb}
		(n+1)c_g(n)=2(2n-1)c_g(n-1)+ (n-1) (2n-1) (2n-3)c_{g-1}(n-2),
	\end{equation}
	where $c_g(n)=0$ for $2g>n$.
\end{theorem}
For genus $0$, the number $c_0(n)$ is given by the Catalan numbers, with generating function ${\bf C}_0(x)= \frac{1-\sqrt{1-4x}}{2x}$. For genus $\geq 1$, the following form for the generating function $\mathbf{C}_g(x)$ is due to \citet{Harer:86} (see also~\citet{Huang:14}).

\begin{theorem}[\citet{Harer:86}]\label{T:c}
	For any $g\geq 1$, the generating function $\mathbf{C}_g(x)$ of linear chord diagrams of genus $g$ is given by
	\[
	\mathbf{C}_g(x)=\sum_{n=2g}^{3g-1} \frac{\kappa_{g}(n) x^{n}} {(1-4x)^{n+\frac{1}{2}}}.
	\]
\end{theorem}

\begin{theorem}[\citet{Li:14,Li-Phd}]
	The numbers $\kappa_{g}(n)$ are positive integers that satisfy an analogue of eq.~(\ref{E:bb})
	\begin{equation}\label{E:cc}
		(n+1)\kappa_{g}(n) =(n-1)(2n-1)(2n-3)
		\kappa_{g-1}(n-2)+ 2 (2n-1) (2n-3)(2n-5)\kappa_{g-1}(n-3),
	\end{equation}
	where $ \kappa_{1}(2)=1$ and $\kappa_{g}(n)=0$ if $n<2g$ or $n>3g-1$.
\end{theorem}
A \emph{shape} is a linear chord diagram  without $1$-arcs in which every stack has length one. For $g\geq 1$, let $s_g(n)$ be the number of shapes of genus $g$ with $n$ arcs
and $\mathbf{S}_g(x)$ denote the corresponding generating polynomial
$\mathbf{S}_g(x)=\sum_{n=2g}^{6g-1} s_g(n) x^n$.
\begin{theorem}[\cite{Huang:14}]\label{T:s}
	For any $g\geq 1$, the generating polynomial of shapes is given by
	\[
	\mathbf{S}_g(x)=\sum_{n=2g}^{3g-1} \kappa_{g}(n)\, x^{n} (1+x)^{n+1}.
	\]
\end{theorem}

{\bf Remark.} Proofs of Theorems~\ref{T:c} and~\ref{T:s} can also be found in ~\citet{Li:14,Li-Phd}.
The notion of shape employed in this paper is slightly different from~\citet{Huang:14} and~\citet{Li:14}.
The notion therein considers a shape to be ``rainbow-free'', where a \emph{rainbow} is an arc connecting
the first and last vertices in a diagram. Allowing for rainbows makes the inflation to a structure more
transparent (Theorem~\ref{T:Dg}) and produces a shape polynomial, $\mathbf{S}_g(x)$, having degree $6g-1$
(Theorem~\ref{T:gasy}).

A \emph{shadow} is a shape without noncrossing arcs. A structure is projected to a shape by deleting all unpaired vertices, iteratively removing all $1$-arcs
and collapsing all stacks to single arcs. A shape can be further projected to a shadow by deleting all noncrossing arcs. These projections from structures to shapes and shadows do not affect genus, see Fig.~\ref{F:sha}.


\begin{figure}
	\centering
	\includegraphics[width=0.8\textwidth]{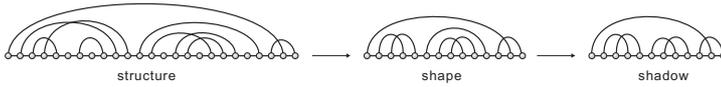}
	\caption
	{\small From structures to shapes and shadows: the projections preserve genus.
	}
	\label{F:sha}
\end{figure}

A linear chord diagram $D$ is called \emph{irreducible}, if and only if for any two arcs, $\alpha_1,\alpha_k$ contained in $D$, there exists a sequence of
arcs $(\alpha_1,\alpha_2,\dots,\alpha_{k-1},\alpha_k)$
such that $(\alpha_i,\alpha_{i+1})$ are crossing. Irreducibility
is equivalent to the concept of primitivity introduced by \cite{Bon:08}.
For arbitrary genus $g$ and
$2g\le\ell\le (6g-2)$, there exists an irreducible shadow of genus
$g$ having exactly $\ell$ arcs~\citep{Huang:11}.

Let $i_g(n)$ denote the number of irreducible shadows of genus $g$
with $n$ arcs, having its generating polynomial $ {\bf I}_g(x)=\sum_{n=2g}^{6g-2}\,i_g(n)x^n$.
For instance for genus $1$ and $2$ we have
\begin{align*}
	{\bf I}_1(x) &= {x}^{2} \left( 1+x \right) ^{2},\\
	{\bf I}_2(x) &= {x}^{4} \left( 1+x \right) ^{4} \left( 17+92\,x+96\,{x}^{2} \right).
\end{align*}
\cite{Han:14} provides a recursion for ${\bf I}_g(x)$.





For RNA secondary structures, the generating function $\mathbf{D}_0 (x,y)$ and its singular expansion are computed in~\citet{Barrett:16}.

\begin{theorem}[\citet{Barrett:16}]\label{T:arcgf}
	For any $\lambda, r\in \mathbb{N}$, the generating function $\mathbf{D}_0 (x,y)$ satisfies the functional equation
	\begin{equation}\label{Eq:arcfe}
		(x^2 y)^r  \mathbf{D}_0 (x,y)^2 -\mathbf{B} (x,y)\, \mathbf{D}_0 (x,y)+\mathbf{A}(x,y)=0,
	\end{equation}
	where
	\begin{align*}
		\mathbf{A}(x,y)&=1-x^2 y+(x^2 y)^r,\\
		\mathbf{B} (x,y)&=(1-x) \mathbf{A}(x,y) + 	(x^2 y)^r\sum_{i=0}^{\lambda-2} x^i.
	\end{align*}
	Explicitly, we have
	\begin{equation}\label{Eq:arcunex}
		\begin{aligned}
			\mathbf{D}_0 (x,y)&= \frac{\mathbf{B} (x,y)-\sqrt{\mathbf{B} (x,y)^2-4 (x^2 y)^r \mathbf{A}(x,y)}}{2 (x^2 y)^r},\\
			\mathbf{D}_0 (x,y)&= \frac{\mathbf{A}(x,y)}{\mathbf{B} (x,y)}\, {\bf
				C}_0\!\left(\frac{(x^2 y)^r \mathbf{A}(x,y)}{\mathbf{B} (x,y)^2}
			\right),
		\end{aligned}
	\end{equation}
	where	${\bf C}_0(x)= \frac{1-\sqrt{1-4x}}{2x}$.
\end{theorem}






\begin{theorem}[\cite{Barrett:16}]\label{T:asyexp}
	For $1\leq \lambda \leq 4$ and $1\leq r \leq 3$, $\mathbf{D}_0 (x,y)$ has the singular expansion
	\begin{equation}\label{Eq:0asy}
		\mathbf{D}_0 (x,y)= \pi(y) + \delta(y)\, \left(\rho(y)-x \right)^{\frac{1}{2}} \left(1+ o(1) \right),
	\end{equation}
	as $x\rightarrow \rho(y)$, uniformly for  $y$ restricted to a neighborhood of $1$, where $\pi(y)$ and $\delta(y)$ are analytic at 1 such that $\delta(1) \neq 0$, and $\rho (y)$ is the minimal positive,
	real solution of
	\[
	\mathbf{B} (x,y)^2 -4 (x^2 y)^r \mathbf{A}(x,y)=0,
	\]
	for $y$ in a neighborhood of $1$. In addition the coefficients of $\mathbf{D}_0 (x,y)$ are asymptotically given by
	\begin{equation} \label{Eq:arcsing}
		[x^{n} ] \mathbf{D}_0 (x,y) =c (y) n^{-\frac{3}{2} } \big(\rho (y)
		\big)^{-n}\ \big(1+O(n^{-1})\big),
	\end{equation}
	as $n \rightarrow \infty$, uniformly, for $y$ restricted to a small neighborhood of $1$, where
	$c (y)$ is continuous and nonzero near $1$.
\end{theorem}






\section{Some Combinatorics}\label{S:GF}

We first derive generating functions $\mathbf{D}_{g}(x,y)$ of topological structures, inflating from shapes.
\begin{theorem}\label{T:Dg}
	Suppose $g,\lambda,r \geq 1$ and $\lambda\leq r+1$. Then the generating function $\mathbf{D}_{g}(x,y)$ is given by
	\begin{equation}\label{Eq:Dg}
		\mathbf{D}_{g}(x,y)=\mathbf{D}_{0}(x,y) \mathbf{S}_{g}\Big(\frac{(x^2 y)^r \mathbf{D}_{0}(x,y)^2}{1-x^2 y-(x^2 y)^r (\mathbf{D}_{0}(x,y)^2-1)}\Big).
	\end{equation}
\end{theorem}

The proof is presented in Section~\ref{S:pf}.

{\bf Remark.}
Theorem~\ref{T:Dg} differs from the results of~\citet{reidys:2013}, which uses an inflation from seeds (linear chord diagrams
in which each stack has length one).  Their method requires additional information on $1$-arcs in seeds, which need to be
related to the generating function $\mathbf{C}_g(x)$ of linear chord diagrams. Shapes can be viewed as seeds without $1$-arcs.
More importantly, there are only finitely many shapes of fixed genus, whence we have a generating polynomial of shapes,
$\mathbf{S}_{g}(x)$, which eventually  explains the subexponential factor in the asymptotic expansion of $\mathbf{D}_{g}(x,y)$
(Theorem~\ref{T:gasy}).

Inflation from shapes to linear chord diagrams implies
\begin{corollary}[\citet{Li:14}]\label{C:CS}
	For $g \geq 1$, we have
	\begin{equation}\label{Eq:CS}
		\mathbf{C}_g(x)=  \mathbf{C}_0(x)\,
		\mathbf{S}_g\!\left(\frac{x \mathbf{C}_0(x)^2}{1-x \mathbf{C}_0(x)^2}\right).
	\end{equation}
\end{corollary}
Now we can derive the functional relation between $\mathbf{D}_{g}(x,y)$ and $\mathbf{C}_{g}(x)$, which generalizes the corresponding results of ~\citet{Barrett:16,reidys:2013}.
\begin{corollary}
	For $g,\lambda,r \geq 1$ and $\lambda\leq r+1$, we have
	\begin{equation}\label{Eq:DS}
		\mathbf{D}_g (x,y)= \frac{\mathbf{A}(x,y)}{\mathbf{B} (x,y)}\, {\bf
			C}_g\!\left(\frac{(x^2 y)^r \mathbf{A}(x,y)}{\mathbf{B} (x,y)^2}
		\right),
	\end{equation}
	where polynomials $\mathbf{A}(x,y)$ and $\mathbf{B} (x,y)$ are defined in Theorem~\ref{T:arcgf}.
\end{corollary}
By interpreting the indeterminant $y$ as a parameter, we consider $\mathbf{D}_g (x,y)$ as a univariate power series and obtain its singular expansion.
\begin{theorem}\label{T:gasy}
	Suppose  $g,\lambda,r \geq 1$ and $\lambda\leq r+1$. Then $\mathbf{D}_g (x,y)$ has the singular expansion
	\begin{equation}\label{Eq:gasy}
		\mathbf{D}_g (x,y)= \delta_g(y)\, \left(\rho(y)-x \right)^{-\frac{6g-1}{2}} \left(1+ o(1) \right),
	\end{equation}
	as $x\rightarrow \rho(y)$, uniformly for  $y$ restricted to a neighborhood of $1$, where $\delta_g(y)$ is analytic at 1 such that $\delta_g(1) \neq 0$,  and $\rho (y)$ is the minimal positive,
	real solution of
	\[
	\mathbf{B} (x,y)^2 -4 (x^2 y)^r \mathbf{A}(x,y)=0,
	\]
	for $y$ in a neighborhood of $1$. In addition the coefficients of $\mathbf{D}_g (x,y)$ are asymptotically given by
	\begin{equation} \label{Eq:coeffg}
		[x^{n} ] \mathbf{D}_g (x,y) =c_g (y) n^{\frac{6g-3}{2} } \big(\rho (y)
		\big)^{-n}\ \big(1+O(n^{-1})\big),
	\end{equation}
	as $n \rightarrow \infty$, uniformly, for $y$ restricted to a small neighborhood of $1$, where
	$c_g (y)$ is continuous and nonzero near $1$.
\end{theorem}

{\bf Remark.}
The dominant singularity of $\mathbf{D}_g (x,y)$ is the same as that of $\mathbf{D}_0 (x,y)$, i.e.~it is independent of genus $g$
and only depends on $r$ and $\lambda$. Furthermore, $\mathbf{D}_g (x,y)$  is the composition of three functions, a polynomial, a
rational function and $\mathbf{D}_0 (x,y)$.
The composition of $\mathbf{D}_0 (x,y)$ with the rational function produces the critical case of singularity analysis~\citep{Flajolet:07a}.
This yields a singular expansion having an exponent $-\frac{1}{2}$ (as opposed to $\frac{1}{2}$ in the case of $\mathbf{D}_0 (x,y)$).
Since the outer function is a polynomial of degree $6g-1$, the singular expansion of $\mathbf{D}_g (x,y)$ has the exponent
$-\frac{6g-1}{2}$, resulting in the subexponential factor $n^{\frac{6g-3}{2} }$.


\section{The Central Limit Theorem}\label{S:PT}


For fixed $\lambda$ and $r$, we  analyze the random variable $\mathbb{Y}_{g,n}$, counting the numbers of arcs in RNA  structures of genus $g$. By construction we have
\[
\mathbb{P}(\mathbb{Y}_{g,n}=l)=\frac{d_{g}(n,l)}{d_{g}(n)},
\]
where $l=2g r,2g r+1,\ldots, \lfloor\frac{n}{2} \rfloor $.
\begin{theorem}\label{T:arcclt}
	For any  $g,\lambda,r \geq 1$ and $\lambda\leq r+1$,	there exists a pair $(\mu,\sigma)$ such that the normalized random variable
	\begin{equation*}
		\mathbb{Y}^{*}_{g,n}=\frac{\mathbb{Y}_{g,n}- \mu \, n}{\sqrt{n\,
				\sigma^2 }},
	\end{equation*}
	converges in distribution to a Gaussian variable with a speed of convergence $O(n^{-\frac{1}{2}})$. That
	is, we have
	\begin{equation}\label{Eq:sclt}
		\lim_{n\to\infty}\mathbb{P}\left(\frac{\mathbb{Y}_{g,n}- \mu
			n}{\sqrt{n\, \sigma^2}} < x \right)  =
		\frac{1}{\sqrt{2\pi}}\int_{-\infty}^{x}\,e^{-\frac{1}{2}t^2} \mathrm{d}t \ ,
	\end{equation}
	where $\mu$ and $\sigma$ are given by
	\begin{equation}\label{Eq:cltpara}
		\mu= -\frac{\theta'(0)}{\theta(0)},
		\qquad \qquad \sigma^2=
		\left(\frac{\theta'(0)}{\theta(0)}
		\right)^2-\frac{\theta''(0)}{\theta(0)},
	\end{equation}
	and $\theta(s) = \rho(e^s)$.
\end{theorem}
Theorem~\ref{T:arcclt} follows directly from Theorem~\ref{T:gasy} and Bender's Theorem in Supplementary Material, setting
$f(x,e^s)=\mathbf{D}_g(x,e^s)$.

In Table~\ref{Tab:clt}, we list the values of $\mu$  for $g\geq 1$, $1\leq \lambda \leq 6$, $1\leq r \leq 6$ and $\lambda\leq r+1$.

{\bf Remark.}  The condition $\lambda\leq r+1$ stems from the same restriction as in Theorem~\ref{T:Dg}. The results of Theorem~\ref{T:Dg} and Theorem~\ref{T:arcclt} can be generalized to the case $\lambda= r+2$ by distinguishing $2$-arcs in shapes, for details see~\citet{Reidys:10} for $k$-noncrossing structures (and also~\citet{Li:11} for RNA-RNA interaction structures). Furthermore, we can establish local limit theorems for the number of arcs in topological RNA structures, similar to~\citet{Jin-Reidys}.

{\bf Remark.} Expectation and variance of the number of arcs in topological RNA structures are independent of genus $g$.

{\bf Remark.} Table~\ref{Tab:clt} shows that  the expectation $\mu$ increases significantly from $r=1$ to $r=2$, indicating that canonical topological structures contain on average more arcs than arbitrary topological structures. The same property on the average number of arcs holds for secondary structures. For  $k$-noncrossing structures, this is not the case, in fact canonical structures contain less arcs~\citep{Reidys:11}.

We find that,
in $r$-canonical RNA structures of genus $g$ with arc-length $\geq \lambda$, the expected number of arcs increases as minimum stack-size $r$ increases or as minimum arc-length $\lambda$ decreases.
\begin{table}[htbp]
	\caption{The central limit theorem for the number of arcs in topological RNA sturctures of genus $g\geq 1$.
		We list the values of $\mu$  derived from eq.~(\ref{Eq:cltpara}). } \label{Tab:clt}
	\begin{center}
		\begin{tabular}{ccccccc}
			\hline
			&{$r=1$} &{$r=2$}
			&{$r=3$} &{$r=4$} &{$r=5$} &{$r=6$} \\
			\hline
			$\lambda=1$ & $0.3333$  & $0.3484$  & $0.3582$  & $0.3651$ & $0.3704$ & $0.3746$ \\
			$\lambda=2$ & $0.2764$  & $0.3172$  & $0.3364$  & $0.3482$ & $0.3565$ & $0.3627$ \\
			$\lambda=3$ &  &   $0.2983$  & $0.3215$  & $0.3358$ & $0.3459$& $0.3534$\\
			$\lambda=4$ &  &  &   $0.3113$ & $0.3268$ & $0.3378$ & $0.3460$ \\
			$\lambda=5$ &  &  &    & $0.3203$ & $0.3316$ & $0.3403$ \\
			$\lambda=6$ &  &  &    &  & $0.3271$ & $0.3359$ \\
			\hline
		\end{tabular}
	\end{center}
\end{table}

\section{Loops in topological RNA structures}\label{S:other}

%

In this section, we apply Theorem~\ref{T:arcclt} and show that all standard loops in RNA structures are asymptotically normal with mean and variance
linear in $n$, independent of genus. We shall study, see Fig.~\ref{F:sloop}:
\begin{itemize}
	\item the number $l^{stack}$ of stacks, i.e.~a maximal sequence of ''parallel'' arcs,
	\item the number $l^{hairpin}$ of hairpin loops. A \emph{hairpin loop} is a pair of the form  $((i,j),[i + 1,j-1])$, where $(i,j)$ is an arc and $[i+1, j-1]$ is an interval, i.e., a sequence of consecutive, unpaired vertices $ i + 1,\ldots,j-1$,
	\item the number $l^{bulge}$ of bulges. A \emph{bulge loop} is either a triple of the form $((i_1,j_1),[i_1 +1, i_2-1],(i_2,j_1-1))$ or $((i_1,j_1),(i_1 +1,j_2),[j_2 +1,j_1-1])$,
	\item the number $l^{interior}$ of interior loops. An \emph{interior loop} is a quadruple $((i_1,j_1),[i_1+1,i_2-1],(i_2,j_2),[j_2+1,j_1-1])$, where $[i_1+1,i_2-1]$ and $[j_2+1,j_1-1]$ are two non-empty intervals, and $(i_2, j_2)$ is nested in $(i_1, j_1)$, i.e., $i_1 < i_2 < j_2 < j_1$,
	\item the number $l^{multi}$ of multi-loops. A \emph{multi-loop} is a sequence $((i_1,j_1),[i_1+1,i_2-1],(i_2,j_2),[j_2+1,i_3-1],(i_3,j_3),[j_3+1,i_4-1],\ldots,(i_k,j_k),[j_k+1,j_1-1])$, for any $k\geq 3 $ and $i_1 < i_2 < j_2< i_3 < j_3\cdots< i_k < j_k < j_1$.
\end{itemize}
\begin{figure}
	\centering
	\includegraphics[width=1\textwidth]{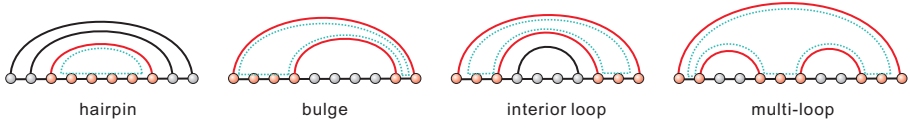}
	\caption
	{\small Hairpin, bulge, interior loop and multi-loop.
	}
	\label{F:sloop}
\end{figure}

For fixed minimum stack-length  $r$ and minimum arc-length $\lambda$, let $d_{g}^{i}(n,l)$  denote the
number of restricted topological RNA structures of $n$ nucleotides and genus $g$, in which 
$i$ marks the loop type.
Let $\mathbf{D}_{g}^{i} (x,y)=\sum_{n,l} d_{g}^{i}(n,l)x^n y^l$  denote the corresponding bivariate generating function.

For secondary structures, the generating functions $\mathbf{D}_{0}^{i} (x,y)$, its asymptotic expansion and the
corresponding central limit theorems  have been studied in~\citet{Hofacker:98,Reidys:11}.
In the following we generalize this to topological structures of genus $g$.

\begin{theorem}\label{T:Dall}
	Suppose $g,\lambda,r \geq 1$ and $\lambda\leq r+1$. Then the generating functions $\mathbf{D}_{g}^{i}(x,y)$ are given by
	\begin{equation*}
		\mathbf{D}_{g}^{i}(x,y)=\mathbf{D}_{0}^{i}(x,y) \mathbf{S}_{g}\Big(h^i(x,y,\mathbf{D}_{0}^{i}(x,y))\Big),
	\end{equation*}
	where $h^i(x,y,z)$ are rational function in $x$, $y$ and $z$ and given in Table~\ref{Tab:h}.
\end{theorem}

\begin{table}[htbp]
	\caption{The rational function $h^i(x,y,z)$. } \label{Tab:h}
	\begin{center}
		\begin{tabular}{ccc}
			\hline\noalign{\smallskip}
			\small \text{stack}   &\small \text{hairpin}  & \small \text{bulge}    \\
			\noalign{\smallskip}\hline\noalign{\smallskip}
			\small$\frac{x^{2r} y z^2}{1-x^2 -x^{2r} y (z^2-1)}$             & \small $\frac{x^{2r}  z^2}{1-x^2 -x^{2r} (z^2-1)}$ 			 &\small $\frac{x^{2r}  z^2}{1-x^2 -x^{2r} (z^2-1-\frac{2x}{1-x} (1-y))}$\\
			\noalign{\smallskip}\hline\noalign{\smallskip}
			\multicolumn{2}{c}{\small \text{interior}} &\small \text{multi}  \\
			\noalign{\smallskip}\hline\noalign{\smallskip}
			\multicolumn{2}{c}{	\small $\frac{x^{2r}  z^2}{1-x^2 -x^{2r} (z^2-1-(\frac{x}{1-x})^2 (1-y))}$}
			&\small $\frac{x^{2r}  z^2}{1-x^2 -x^{2r} (y(z^2-1)+\frac{x(2-x)}{(1-x)^2} (1-y))}$\\
			\noalign{\smallskip}\hline\noalign{\smallskip}
		\end{tabular}
	\end{center}
\end{table}

As in  Theorem~\ref{T:gasy} and Theorem~\ref{T:arcclt}, we next compute the singular expansion of $\mathbf{D}_{g}^{i}(x,y)$,
the asymptotics of its coefficients and the limit theorems as follows:
\begin{theorem}\label{T:gasyall}
	Suppose  $g,\lambda,r \geq 1$ and $\lambda\leq r+1$. Then $\mathbf{D}_g^{i} (x,y)$ has the singular expansion
	\begin{equation}\label{Eq:gasyall}
		\mathbf{D}_g^{i} (x,y)= \delta_g(y)\, \left(\rho_i(y)-x \right)^{-\frac{6g-1}{2}} \left(1+ o(1) \right),
	\end{equation}
	as $x\rightarrow \rho(y)$, uniformly for  $y$ restricted to a neighborhood of $1$, where $\rho_i (y)$ is the dominant singularity of $\mathbf{D}_0^{i} (x,y)$ for secondary structures. In addition the coefficients of $\mathbf{D}_g^{i} (x,y)$ are asymptotically given by
	\begin{equation} \label{Eq:coeffgall}
		[x^{n} ] \mathbf{D}_g^{i} (x,y) =c_g (y) n^{\frac{6g-3}{2} } \big(\rho_i (y)
		\big)^{-n}\ \big(1+O(n^{-1})\big),
	\end{equation}
	as $n \rightarrow \infty$, uniformly for $y$ restricted to a small neighborhood of $1$.
\end{theorem}

\begin{theorem}\label{T:arccltall}
	For any  $g,\lambda,r \geq 1$ and $\lambda\leq r+1$, the distribution of
	stacks, hairpin loops, bulge loops, interior loops and multi-loops in genus $g$ structures  satisfies the central and local limit theorem with mean $\mu n$ and variance $\sigma n$, where $\mu$ and $\sigma$ are the same as those of secondary structures.	
\end{theorem}


In difference to the case of arcs, the correlation between the means of these parameters and the minimum stack-size is as follows:
\begin{corollary}\label{C:otherstack}
	Suppose $g\geq 1$, $1\leq \lambda \leq 6$, $1\leq r \leq 6$ and $\lambda\leq r+1$. In $r$-canonical RNA structures of genus $g$ with arc-length $\geq \lambda$, the expectation of the number of stacks, hairpin loops, bulge loops, etc, decreases as minimum stack-size $r$ increases or as minimum arc-length $\lambda$ increases.
\end{corollary}


\section{Block decomposition}\label{S:loop}

In this section we introduce the block decomposition of a topological RNA structure.

Given an RNA structure, $S$, its arc-set induces the line graph $\Sigma(S)$~\citep{Whitney:32}, obtained by mapping each arc $\alpha$ into the
vertex $\Sigma(\alpha)=v_\alpha$ and connecting any two such vertices iff their corresponding arcs are crossing in $S$,
$\Sigma\colon S \to \Sigma(S)$.
An \emph{arc-component} of a structure $S$ is the diagram consisting of a set of arcs $A$ such that $\Sigma(A)$ is a component in $\Sigma(S)$.
The arc-component containing only one arc is called \emph{trivial}.
Let $S\{i,j\}$ denote the arc-component with the left- and rightmost endpoints $i$ and $j$.
By construction, for any two different arc-components $S\{i_1,j_1\}$ and $S\{i_2,j_2\}$ with $i_1<i_2$, we have either $i_1<j_1<i_2<j_2$ or
$i_1<i_2<j_2<j_1$.



An unpaired vertex $k$ is said to be \emph{interior} to the arc-component $S\{i,j\}$ if $i<k<j$.  If there is no other
arc-component $S\{p,q\}$  such that $i<p<k<q<j$, we call $k$ \emph{immediately interior} to $S\{i,j\}$.
An \emph{exterior vertex} is an unpaired vertex which is not interior to any arc-component. A \emph{block} is an arc-component
together with all its immediately interior vertices.

Each paired vertex is contained in a base pair, belonging to a unique  arc-component, while an unpaired vertex is either exterior
or interior to a unique arc-component. Each block is characterized by its arc-component, whence any topological RNA structure can
be uniquely decomposed into blocks and exterior vertices, see Fig.~\ref{F:loopdec}.

\begin{figure}
	\centering
	\includegraphics[width=1\textwidth]{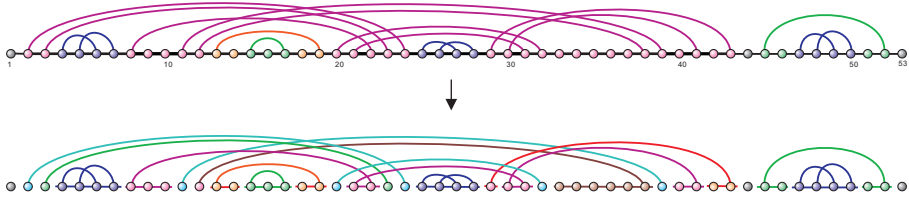}
	\caption
	{\small The block decomposition of an RNA structure. The structure  is decomposed into blocks and exterior vertices.
	}
	\label{F:loopdec}
\end{figure}

A block is call \emph{irreducible} if its arc-component contains at least two arcs.
A block with endpoints $i_2$ and $j_2$ is  \emph{nested} in a block with endpoints  $i_1$ and $j_1$   if $i_1<i_2<j_2<j_1$.
Removing  all the blocks nested in an irreducible block could induce stacks.
As shown by \cite{Han:14,Li:13}, an irreducible block is projected to an irreducible shadow by removing all nested blocks
and interior vertices, and collapsing its original stacks and induced stacks into single arcs, see Fig.~\ref{F:induce}.

\begin{figure}
	\centering
	\includegraphics[width=1\textwidth]{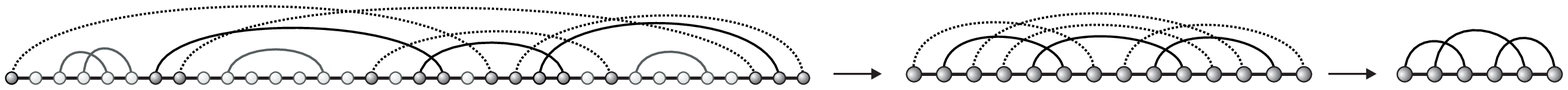}
	\caption
	    {\small  From an irreducible block to its irreducible shadow. Nested blocks and interior vertices are removed and
              all stacks are collapsed.
	}
	\label{F:induce}
\end{figure}

For secondary structures, each arc-component consists of a single arc. Thus each block is trivial and we immediately observe that they organize
as stacks, hairpin loops, bulges, interior loops and multi-loops, depending on  the number of interior vertices and nested blocks. Therefore,
the block decomposition coincides with the standard decomposition of secondary structures into stacks, loops and exterior vertices~\citep{Waterman:78s}.
In this case, a loop naturally corresponds to a boundary component in the fatgraph model of secondary structures.


However, boundary components in pk-structures can traverse different blocks. Moreover, these boundaries intertwine, see Fig.~\ref{F:g1irr}.
The irreducible shadow of a block can be viewed as a combination of intertwined boundary components, see Fig.~\ref{F:g1irr}.
\begin{figure}
	\centering
	\includegraphics[width=1\textwidth]{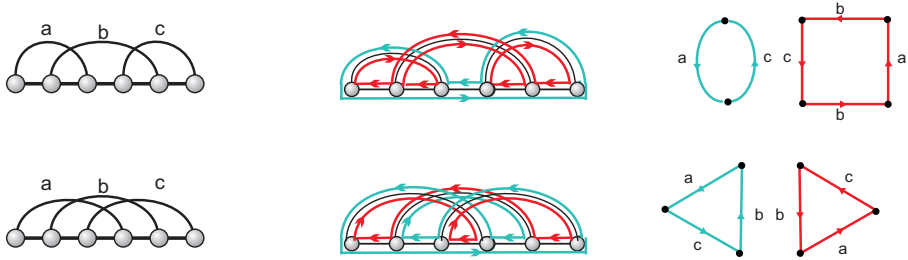}
	\caption
	{\small Irreducible shadows of genus 1, represented as intertwining boundary components.
	}
	\label{F:g1irr}
\end{figure}

\section{Irreducible shadows}\label{S:ploop}

%
In this section, we have a closer look at irreducible shadows. 
By definition, an irreducible block is characterized by an irreducible shadow that appears in the block decomposition of a topological RNA structure. The four irreducible shadows of genus one  naturally correspond to  four types of pseudoknots. They are H-type pseudoknots, kissing hairpins (K-type), $3$-knots (L-type) and $4$-knots (M-type),  see Fig.~\ref{F:genus1} and Supplementary Material. 

\begin{figure}
	\centering
	\includegraphics[width=1\textwidth]{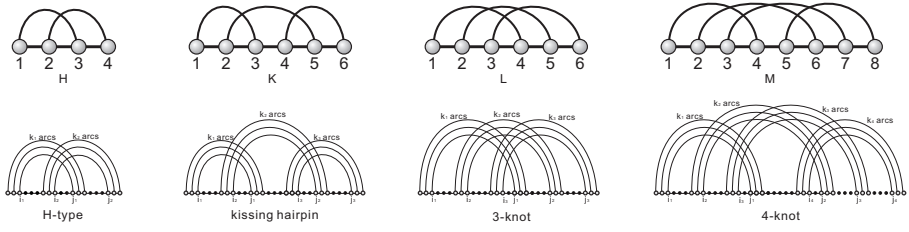}
	\caption
	{\small The four irreducible shadows of genus $1$ and four types of pseudoknots.
	}
	\label{F:genus1}
\end{figure}

To investigate the distribution of pseudoknots in structures of genus $g$, we consider the inflation from shapes to structures
marking these pseudoknots. To this end we introduce the bivariate generating functions for the generating functions of shapes
and structures of genus $g$. For $I\in \{H,K,L,M\}$, let $s_g^{I} (n,l^I)$ and $d_g^{I} (n,l^I)$ denote the number  of genus $g$ shapes
and structures having $l^I$ $I$-type pseudoknots. Their corresponding generating function are denoted by $\mathbf{S}_g^{I} (x,y)$
and $\mathbf{D}_g^{I} (x,y)$. For example for $I=H$ and $g=1,2$, we compute
\begin{align*}
	\mathbf{S}_1^{H} (x,y) &=x^3 (x+1) (x+2)+ {x}^{2}(1+x) y,\\
	\mathbf{S}_2^{H} (x,y) &= {x}^{4} \left( 1+x \right) ^{2} \Big( 17+143 x +447 x^2+637 x^3+420 x^4+105 x^5\\
	&+\left(20x +36 x^2+ 14 x^3 \right) y+(4+5 x) y^2\Big).
\end{align*}

The computation of $\mathbf{S}_g^{I} (x,y)$ for arbitrary genus is obtained as follows: in the block decomposition of the shape,
each $I$-type pseudoknot corresponds to an irreducible block of the respective type. To count these pseudoknots, we only need
to enumerate the number of the corresponding irreducible shadows in the decomposition.
To this end we define the bivariate generating function $ {\bf I}^I_g(x,y)$ of irreducible shadows of genus $g$, where $y$
marks the number of $I$-type pseudoknots, where $I\in \{H,K,L,M\}$. Then
\begin{align*}
	{\bf I}^H_1(x,y) &= {x}^{2}y  + 2 x^3+x^4 \\
	{\bf I}^K_1(x,y)={\bf I}^L_1(x,y) &= {x}^{3}y + x^2+ x^3+x^4\\
	{\bf I}^M_1(x,y)&= {x}^{4}y + x^2 + 2 x^3\\
	{\bf I}^I_g(x,y)&={\bf I}_g(x) \qquad \text{for any } g\geq 2 \text{ and } I\in \{H,K,L,M\} .
\end{align*}

The trivariate generating functions of irreducible shadows and shapes of genus $g$
with $n$ arcs and $l^I$ pseudoknots of type $I$ are denoted by
\begin{align*}
	{\bf I}^I(x,y,t)&=\sum_{g\geq 1}{\bf I}^I_g(x,y)\,t^g,\\
	{\bf S}^I(x,y,t)&=1+\sum_{g\geq 1}{\bf S}^I_g(x,y)\,t^g=1+\sum_{g\geq 1}\sum_{n=2g}^{6g-1}\,s_g^{I} (n,l^I) \, x^n y^{l^I} t^g .
\end{align*}

We derive
\begin{theorem}\label{T:SIrelation}
	For any type $I\in \{H,K,L,M\}$, the generating functions  ${\bf S}^I(x,y,t)$ and ${\bf I}^I(x,y,t)$ satisfy
	\begin{equation}\label{Eq:SIrelation}
		{\bf S}^I(x,y,t) =1+ x \left({\bf S}^I(x,y,t)-1\right)^2+(x+1){\bf I}^I\left(\frac{x{\bf S}^I(x,y,t)^2 }{1-x \left({\bf S}^I(x,y,t)^2-1\right)},y,t\right) .
	\end{equation}
\end{theorem}

Theorem~\ref{T:SIrelation} implies via exacting the coefficient of $t^g$ on both sides of eq.~(\ref{Eq:SIrelation})
a recursion allowing the computation of ${\bf S}^I_g(x,y)$ for any given genus:
\begin{corollary}\label{C:shape}
	For $g\ge 1$, ${\bf S}^I_g(x,y)$ satisfies the following recursion
	\begin{equation}\label{Eq:shaperec}
	\begin{aligned}
		{\bf S}^I_g(x,y)= &x \sum_{i=1}^{g-1} {\bf S}^I_i(x,y){\bf S}^I_{g-i}(x,y)\\
		&+(x+1)
		\sum_{j=1}^{g} [t^{g-j}]{\bf I}_j^I\left(\frac{x\,\left(\sum_{k=0}^{g-j}{\bf S}^I_k(x,y) t^k\right)^2}{1-x\,\left(\left(\sum_{k=0}^{g-j}{\bf S}^I_k(x,y) t^k\right)^2-1\right)},y\right),
	\end{aligned}
	\end{equation}
	where ${\bf S}^I_0(x,y)=1$.
\end{corollary}
The polynomials ${\bf S}^H_g(x,y)$ for $1\leq g\leq 4$ are listed in Supplementary Material.

\begin{corollary}\label{C:shape2}
	For fixed genus $g\geq 1$, a shape containing at least one H-type (K-, L- or M- types) pseudoknot has at most $6g-3$ arcs ($6g-2$, $6g-2$ or $6g-1$). This bound is sharp, i.e., there exist shapes with $6g-3$ arcs ($6g-2$, $6g-2$ or $6g-1$) containing one H-type (K-, L- or M- types) pseudoknot.
\end{corollary}
The result clearly holds for shapes of genus one and an inductive argument utilizing eq.~(\ref{Eq:shaperec}) for the terms with a positive
exponent of $y$, yields this result.


For fixed $\lambda$ and $r$, we are now in position to analyze the random variable $\mathbb{X}^{I}_{g,n}$, counting the numbers of $I$-type
pseudoknots in RNA  structures of genus $g$, where $I\in \{H,K,L,M\}$. By construction we have $\mathbb{P}(\mathbb{X}^I_{g,n}=l)=\frac{d^I_{g}(n,l)}{d_{g}(n)}$,
where $l=0,1,\ldots, g $.

Next we compute the expectation of the number of these pseudoknots in genus $g$ structures. 


\begin{theorem}\label{T:loops}
	For any $g\geq 1$, the expectation of H-, K-, L- and M-type pseudoknots in uniformly generated  structures of genus $g$ is $ O\left(n^{-1}\right)$, $ O(n^{-\frac{1}{2}})$, $ O(n^{-\frac{1}{2}})$ and $ O(1)$, respectively.
\end{theorem}

\begin{corollary}\label{C:loops}
	In particular, for $g=1$, $\lambda=1$ and $r=1$, we have
	\begin{align*}
		\mathbb{P}(\mathbb{X}^H_{1,n} =1) &= \frac{288+O\left(n^{-1}\right)}{16 n-51+O\left(n^{-1}\right)},\\	
		\mathbb{P}(\mathbb{X}^K_{1,n} =1)=\mathbb{P}(\mathbb{X}^L_{1,n} =1) &= \frac{24 \left(\sqrt{3 \pi } \sqrt{n}-18\right)+O\left(n^{-1}\right)}{16
			n-51+O\left(n^{-1}\right)},\\
		\mathbb{P}(\mathbb{X}^M_{1,n} =1) &= \frac{16 n-48 \sqrt{3 \pi } \sqrt{n}+525+O\left(n^{-1}\right)}{16 n-51+O\left(n^{-1}\right)}.	
	\end{align*}	
\end{corollary}

{\bf Remark:} this approach allows to obtain the statistics for any irreducible shadow.


\section{Discussion}\label{S:Dis}

%
The backbone of this paper is a novel bivariate generating function, that is closely related to the
shape polynomial. Our approach makes evident which properties of RNA structures are not affected by
increasing their complexity, stipulating that complexity is tantamount to the topological genus of
the surface needed to embed them without crossings.

The singularity analysis of this generating function in Section~\ref{S:pf} allows to obtain analytically
a plethora of limit distributions and we put these into context with data from uniformly sampled
topological structures.

One class of results is centered around the fact that topological genus does only enter the subexponential
factors of the singular expansion. It implies that all basic loop types appear with the same frequencies as
in RNA secondary structures, i.e.~topological structures of genus zero.
In Fig.~\ref{F:uniform}, we present the distribution of the number of arcs, obtained by uniformly
sampling $10^5$ structures over $100$ nucleotides of genus $0$, $1$ and $2$~\citep{Huang:13},
respectively. As predicted by Theorem~\ref{T:arcclt} the distributions are not affected by genus. 
This finding holds also for any other loop as shown via Theorem~\ref{T:arccltall}.

\begin{figure}
	\centering
	\includegraphics[width=0.8\textwidth]{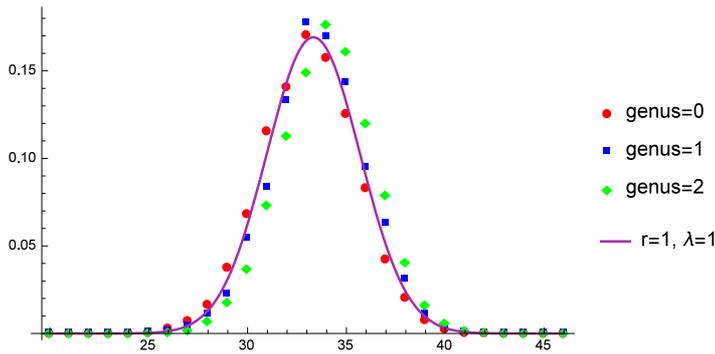}
	\caption
	    {\small The number of arcs in RNA structures: we contrast the limit distribution of RNA structures
              having both minimum arc-length and stack-length one (solid lines) with the distribution of the
              uniformly sampled structures of length $100$ having genus $0$, $1$ and $2$~\citep{Huang:13}.
	}
	\label{F:uniform}
\end{figure}

Furthermore our singular expansion shows that the exponential growth rate is affected by minimum arc-length and stack-length
in an interesting way: for fixed genus, $g$, canonical structures of genus $g$ contain on average {\it more} arcs than arbitrary
structures of genus $g$, but {\it fewer} other structural features, such as stacks, hairpins, bulges, interior loops and multi-loops.


The second class of results is centered around enhancing the singular expansions in such a way that they can ``detect''
novel features. We chose to consider here features previously studied, i.e.~H-type pseudoknots, kissing hairpins and 3-knots. In Section~\ref{S:ploop} we give a detailed analysis that can easily be tailored to detect any other feature.

For the above mentioned features we contrast Theorem~\ref{T:loops} with data obtained from a uniform sample of $10^5$
structures over $500$ nucleotides of genus $1$, see Table~\ref{Tab:loops}. The table shows that the predicted expectation
values in the limit of long sequences are consistent with the data obtained from the uniform sampling.
This insight into features of random structures allows to evaluate the data on RNA pk-structures contained in data bases
such as Nucleic Acid Database (NDB)~\citep{Berman:92,Narayanan:14}. Namely, irrespective of genus, we find in data bases a dominance of H-types, while kissing hairpins and 3-knots are
rather rare. This is quite the opposite to what happens in random structures and suggests that H-types have distinct energetic
advantages. Of course this assumes that H-types have been correctly identified and no additional bonds have been missed.

In Fig.~\ref{F:p-Loop3}, we display the expectation of H-types, kissing hairpins, $3$-knots and $4$-knots in uniformly generated
RNA structures of genus $2$ together with our theoretical estimates derived in Theorem~\ref{T:loops}.
As predicted by Theorem~\ref{T:loops}, H-types, kissing hairpins and $3$-knots appear at the rate $O\left(n^{-1}\right)$ and
$O(n^{-\frac{1}{2}})$, respectively. Only $4$-knots exhibit a rate independent of sequence length.
These findings are in line with Corollary~\ref{C:shape2},  since only shapes containing a $4$-knot type can have the
maximum number of arcs and shapes containing H-types, kissing hairpins or $3$-knots lack at least one arc.
In addition, close inspection of the proof of Theorem~\ref{T:SIrelation} allows to construct all shapes containing a specific
feature.

\begin{table}
	\caption{The number of different types of pseudoknots in RNA structures: we contrast our
	  estimates in Theorem~\ref{T:loops} and the exact enumerations in a uniform sampling
          of  $10^5$ structures of length $500$ having genus $1$~\citep{Huang:13}.} \label{Tab:loops}
	\begin{tabular}{ccccc}
		\hline\noalign{\smallskip}
		\small   &\small H-type  &\small kissing hairpin  &\small $3$-knot & \small $4$-knot     \\
		\noalign{\smallskip}\hline\noalign{\smallskip}
		\small\text{sample}   &\small$331$              & \small$1563$ & \small $1587$ &\small $6519$	 \\
		\small\text{estimate}   &\small$362.3$               & \small$1529.2$ & \small $1529.2$ &\small $6579.4$ \\
		\noalign{\smallskip}\hline
	\end{tabular}
\end{table}

\begin{figure}
	\centering
	\includegraphics[width=.9\textwidth]{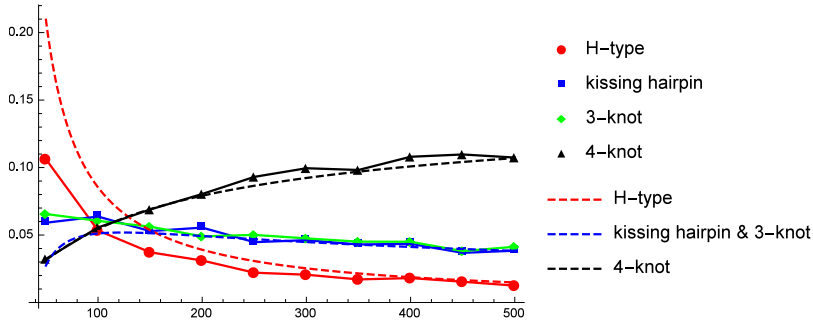}
	\caption
	    {\small The expectation (solid) and theoretical estimates (dashed) of H-types, kissing hairpins, $3$-knots
              and $4$-knots in uniformly generated RNA structures of genus $2$ as a function of sequence length $n$. The dashed curves are analytically computed along the lines of Corollary~\ref{C:loops}.
 	}
	\label{F:p-Loop3}
\end{figure}

There is only very limited experimental data on the energy of RNA pseudoknots. The established notion of
how to compute energies for hairpin-, interior- and multi-loops combined with our notion of irreducible
shadows as building blocks, allows to derive such energies in topological structures.
The energy model for pseudoknots can then be developed in an analogous way and this is work in progress and will be
reported elsewhere.


\section{Proofs}\label{S:pf}

%

\begin{proof}[Proof of Theorem~\ref{T:Dg}]
	Let $\mathcal{D}_{g}(n,l)$ denote the collection of $r$-canonical topological RNA structures of $n$ vertices, $l$ arcs and genus $g$, with minimum arc-length $\lambda$. Let furthermore $\mathcal{S}_{g}(k)$ denote the collection of shapes of $k$ arcs and genus $g$. There is a natural projection $\varphi$ from RNA structures to shapes defined by first removing all secondary structures and then collapsing each stack into a single arc
	\[
	\varphi: \; \cup_{n\geq1} \cup_{l\geq 0} \mathcal{D}_{g}(n,l) \rightarrow  \cup_{k\geq1} \mathcal{S}_{g}(k).
	\]
	It is clear that $\varphi$ is surjective and preserves genus. For any shape $\omega\in\mathcal{S}_{g}(k) $, let
	\[
	\mathcal{D}_{g}^{\omega}(n,l) = \mathcal{D}_{g}(n,l) \cap \varphi^{-1}(\omega)
	\]
	denote the subset of the fiber $\varphi^{-1}(\omega)$ and let  $\mathbf{D}_{g}^{\omega}(x,y)$ be its generating function.
	
	We first prove that for any shape $\omega$ of genus $g$ and $k$ arcs
	\begin{equation}\label{Eq:omega}
	\mathbf{D}_{g}^{\omega}(x,y)=\Bigg(\frac{(x^2 y)^r }{1-x^2 y-(x^2 y)^r (\mathbf{D}_{0}(x,y)^2-1)}\Bigg)^k \,\mathbf{D}_{0}(x,y)^{2k+1}.
	\end{equation}
	
	We adopt the following notations from symbolic enumeration~\citep{Flajolet:07a}. Let $=$ denote set-theoretic bijection, $+$ disjoint union, $\times$ Cartesian product, $\mathcal{I}$ the collection containing only one element of size $0$, and the sequence construction $\textsc{Seq}(\mathcal{A}):={\mathcal{I}}+\mathcal{A}+(\mathcal{A}\times\mathcal{A})+
	(\mathcal{A}\times\mathcal{A}\times\mathcal{A})+\cdots$ for any collection $\mathcal{A}$.
	
	We shall construct $\cup_{n\geq1} \cup_{l\geq 0} \mathcal{D}_{g}^{\omega}(n,l)$ via the inflation from a shape $\omega$ using simple combinatorial building blocks such as arcs $\mathcal{R}$, stacks	$\mathcal{K}$, induced	stacks $\mathcal{N}$, stems $\mathcal{M}$, and secondary structures $\mathcal{D}_{0}$. We inflate the shape $ \omega\in\mathcal{S}_{g}(k)$ to a $\mathcal{D}_{g}^{\omega}$-structure in
	two steps.
	
	{\bf Step I:} we inflate any arc in $\omega$ to a stack of
	length at least $r$ and subsequently add additional stacks. The
	latter are called induced stacks and have to be separated by means
	of inserting secondary structures, see Fig.~\ref{F:inflate}.
	\begin{figure}
		\centering
		\includegraphics[width=1\textwidth]{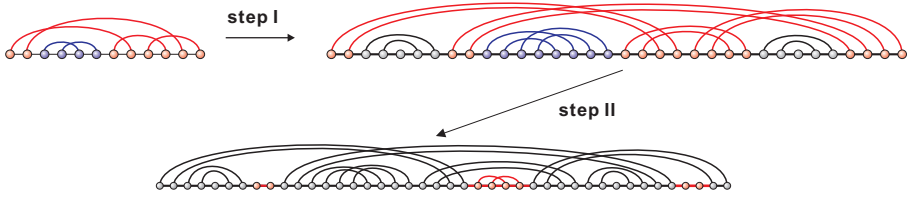}
		\caption
		{\small From a shape to a structure. Step I: inflation of each arc to  a stem. Step II: insertion of additional secondary structures.
		}
		\label{F:inflate}
	\end{figure}
	
	Note that during this first inflation step no secondary structures,
	other than those necessary for separating the nested stacks are
	inserted. We generate
	\begin{itemize}
		\item secondary structures $\mathcal{D}_{0}$ with stack-length $\geq r$ and arc-length $\geq \lambda$
		having the generating function $\mathbf{D}_{0}(x,y)$,
		\item  arcs $\mathcal{R}$ with generating function ${\bf R}(x,y)=x^2 y$,
		\item stacks, i.e.~pairs consisting of the minimal sequence of arcs
		$\mathcal{R}^r$ and an arbitrary extension consisting of
		arcs of arbitrary finite length
		$\mathcal{K}=\mathcal{R}^{r}\times\textsc{Seq}\left(\mathcal{R}\right)$
		having the generating function
		\begin{equation*}
		\mathbf{K}(x,y)  = \frac{ (x^2 y)^{r}}{1-x^2 y},
		\end{equation*}
		\item induced stacks, i.e.~stacks together with at least one secondary structure on
		either or both of its sides,
		$\mathcal{N}=\mathcal{K}\times\left( \mathcal{D}_{0}^2 -1 \right),$
		having the generating function
		\begin{equation*}
		\mathbf{N}(x,y)=\frac{ (x^2 y)^{r}}{1-x^2 y} \left( \mathbf{D}_{0}(x,y)^2-1\right),
		\end{equation*}
		\item stems, that is pairs consisting of stacks $\mathcal{K}$
		and an arbitrarily long sequence of induced stacks
		$		\mathcal{M}=\mathcal{K} \times \textsc{Seq}\left(\mathcal{N}\right), $
		having the generating function
		\begin{align*}
		\mathbf{M}(x,y)&=\frac{\mathbf{K}(x,y)}{1-\mathbf{N}(x,y)}\\
		&=\frac{\frac{ (x^2 y)^{r}}{1-x^2 y}}
		{1-\frac{ (x^2 y)^{r}}{1-x^2 y} \left( \mathbf{D}_{0}(x,y)^2-1\right)}\\
		&=\frac{(x^2 y)^r }{1-x^2 y-(x^2 y)^r (\mathbf{D}_{0}(x,y)^2-1)}.
		\end{align*}
	\end{itemize}
	By inflating each arc into a stem, we have the corresponding generating function
	\begin{equation*}
	\mathbf{M}(x,y)^{k}=\Bigg(\frac{(x^2 y)^r }{1-x^2 y-(x^2 y)^r (\mathbf{D}_{0}(x,y)^2-1)}\Bigg)^k.
	\end{equation*}
	
	{\bf Step II:} here we insert additional secondary structures at the
	remaining $2 k+1$ possible locations, see Fig.~\ref{F:inflate}.
	Formally, these insertions are expressed via the combinatorial class
	$\mathcal{D}_{0}^{2 k+1}$
	whose generating function is $\mathbf{D}_{0}(x,y)^{2 k+1}.$
	
	Notice that any stack, from either  inflation of arcs or  inserted secondary structures, has length at least $r$. Any non-crossing arc has length at least $\lambda$ since it must be located in some inserted secondary structure with arc-length $\geq \lambda$. The key point here is that the restriction $\lambda \leq r+1$ guarantees that any crossing arc has after inflation  a minimum arc-length of $r+1\geq \lambda$. Therefore, combining Step I and Step II, we create a $\mathcal{D}_{g}^{\omega}$-structure with stack-length $\geq r$ and arc-length $\geq \lambda$,  and  arrive at
	\[	\cup_{n\geq1} \cup_{l\geq 0} \mathcal{D}_{g}^{\omega}(n,l)=\mathcal{M}^k\times \mathcal{D}_{0}^{2 k+1}.
	\]
	Accordingly	
	\begin{align*}
	\mathbf{D}_{g}^{\omega}(x,y)&=\mathbf{M}(x,y)^{k}\mathbf{D}_{0}(x,y)^{2 k+1}\\
	&=\Bigg(\frac{(x^2 y)^r }{1-x^2 y-(x^2 y)^r (\mathbf{D}_{0}(x,y)^2-1)}\Bigg)^k \,\mathbf{D}_{0}(x,y)^{2k+1},
	\end{align*}
	whence eq.~(\ref{Eq:omega}).
	
	Now we are in position to prove eq.~(\ref{Eq:Dg}). Since each $ \mathcal{D}_{g}(n,l)$-structure projects, via $\varphi$,  to a unique shape $\omega$ of genus $g$ and $k$ arcs, we have
	\[
	\cup_{n\geq1} \cup_{l\geq 0} \mathcal{D}_{g}(n,l)= \cup_{n\geq1} \cup_{l\geq 0} \cup_{\omega} \mathcal{D}_{g}^{\omega}(n,l),
	\]
	i.e. we have
	\[
	\mathbf{D}_{g}(x,y)=\sum_{\omega}\mathbf{D}_{g}^{\omega}(x,y).
	\]
	According to eq.~(\ref{Eq:omega}), $\mathbf{D}_{g}^{\omega}(x,y)$ only depends on the number of arcs in the shape $\omega$, and  we can therefore derive
	\begin{align*}
	\mathbf{D}_{g}(x,y)&=\sum_{k\geq 1} \sum_{\omega\in\mathcal{S}_{g}(k) }\mathbf{D}_{g}^{\omega}(x,y)\\
	&=\sum_{k\geq 1} s_g(k)\mathbf{D}_{g}^{\omega}(x,y)\\
	&=\sum_{k\geq 1} s_g(k) \Bigg(\frac{(x^2 y)^r }{1-x^2 y-(x^2 y)^r (\mathbf{D}_{0}(x,y)^2-1)}\Bigg)^k \,\mathbf{D}_{0}(x,y)^{2k+1}\\
	&=\mathbf{D}_{0}(x,y)\sum_{k\geq 1} s_g(k) \Bigg(\frac{(x^2 y)^r \mathbf{D}_{0}(x,y)^2 }{1-x^2 y-(x^2 y)^r (\mathbf{D}_{0}(x,y)^2-1)}\Bigg)^k\\
	&=\mathbf{D}_{0}(x,y) \mathbf{S}_{g}\Big(\frac{(x^2 y)^r \mathbf{D}_{0}(x,y)^2}{1-x^2 y-(x^2 y)^r (\mathbf{D}_{0}(x,y)^2-1)}\Big).
	\end{align*}
\end{proof}
\begin{proof}[Proof of Theorem~\ref{T:gasy}]
	To prove eq.~(\ref{Eq:gasy}), we first determine the location of the dominant singularity of $\mathbf{D}_g (x,y)$.  Set
	\[
	h(x,y,z)=	\frac{(x^2 y)^r z^2}{1-x^2 y-(x^2 y)^r (z^2-1)}.
	\]
	In view of eq.~(\ref{Eq:Dg}) of Theorem~\ref{T:Dg}, we have
	\begin{equation}\label{Eq:Dg2}
	\mathbf{D}_g (x,y)=\mathbf{D}_{0}(x,y) \mathbf{S}_{g}(h(x,y,\mathbf{D}_0 (x,y))).
	\end{equation}
	Since $\mathbf{S}_{g}(x)$ is a polynomial in $x$, the singularities of $\mathbf{D}_g (x,y)$ are those of $\mathbf{D}_0(x,y)$ and  those of $h(x,y,\mathbf{D}_0 (x,y))$. By Theorem~\ref{T:asyexp}, we know that the dominant singularity of $\mathbf{D}_0(x,y)$ is given by $\rho (y)$, the minimal positive, real solution of
	\[
	\mathbf{B} (x,y)^2 -4 (x^2 y)^r \mathbf{A}(x,y)=0,
	\]
	for $y$ in a neighborhood of $1$. For the denominator of $h(x,y,\mathbf{D}_0 (x,y))$, we compute
	\begin{equation}\label{Eq:dh}
	\begin{aligned}
	&1-x^2 y-(x^2 y)^r (\mathbf{D}_0 (x,y)^2-1) \\
	=  &\mathbf{A}(x,y)-(x^2 y)^r \mathbf{D}_0 (x,y)^2\\
	=  &\frac{4 (x^2 y)^r\mathbf{A}(x,y)- \Big(\mathbf{B} (x,y)- \sqrt{\mathbf{B} (x,y)^2 -4 (x^2 y)^r \mathbf{A}(x,y)}\Big)^2}{4 (x^2 y)^r}  \\
	=  &\frac{\sqrt{\mathbf{B} (x,y)^2 -4 (x^2 y)^r \mathbf{A}(x,y) }\Big(\mathbf{B} (x,y)- \sqrt{\mathbf{B} (x,y)^2 -4 (x^2 y)^r \mathbf{A}(x,y)}\Big)}{2 (x^2 y)^r}  \\
	=  &\sqrt{\mathbf{B} (x,y)^2 -4 (x^2 y)^r \mathbf{A}(x,y) } \; \mathbf{D}_0 (x,y),
	\end{aligned}
	\end{equation}
	where the first equality uses the definition of $\mathbf{A}(x,y)$ in Theorem~\ref{T:arcgf}, and the second and fourth equalities follow from eq.~(\ref{Eq:arcunex}) of Theorem~\ref{T:arcgf}.
	The above computation implies that the dominant singularity of $h(x,y,\mathbf{D}_0 (x,y))$ is given by  $\rho (y)$. Consequently, $\mathbf{D}_g (x,y)$ has the unique dominant singularity $\rho (y)$.
	
	{\bf Claim 1.} There exists some  function $k(y)$ such that  for $x \rightarrow \rho(y)$, the singular expansion of $h(x,y,\mathbf{D}_0 (x,y))$ is given by
	\begin{equation}\label{Eq:h}
	h(x,y,\mathbf{D}_0 (x,y))= k(y)\, \left(\rho (y)-x\right)^{-\frac{1}{2}} \left(1+ o(1) \right),
	\end{equation}
	uniformly for  $y$ restricted to a neighborhood of $1$, where $k(y)$ is analytic for  $y$ in a neighborhood of $1$ and $k(1)\neq 0$.
	
	From the above computation, it is crucial to notice that, for $h(x,y,\mathbf{D}_0 (x,y))$ viewed as a composition of two functions $h(x,y,z)$ and  $\mathbf{D}_0 (x,y)$, we can use a phenomenon known as a confluence of singularities of the internal and external functions (the critical paradigm, see ~\citet{Flajolet:07a}). It basically means that the type of the singularity is determined by a mix of the types  of the internal and external functions. To prove  Claim 1,
	we first rewrite $h(x,y,\mathbf{D}_0 (x,y))$ using eq.~(\ref{Eq:dh})
	\begin{align*}
	h(x,y,\mathbf{D}_0 (x,y)) &= \frac{(x^2 y)^r \mathbf{D}_{0}(x,y)^2}{1-x^2 y-(x^2 y)^r (\mathbf{D}_{0}(x,y)^2-1)}\\
	&= \frac{(x^2 y)^r \mathbf{D}_{0}(x,y)}{\sqrt{\mathbf{B} (x,y)^2 -4 (x^2 y)^r \mathbf{A}(x,y) }}.
	\end{align*}
	Notice that $\sqrt{\mathbf{B} (x,y)^2 -4 (x^2 y)^r \mathbf{A}(x,y) }= \psi(x,y) \left(\rho (y)-x\right)^{\frac{1}{2}}  $, where $\psi(x,y)$ is analytic for $x$ in a neighborhood of $\rho(y)$ and $y$ in a neighborhood of $1$ such that $\psi(\rho(y),y)\neq 0$ for $y$ in a neighborhood of $1$. By setting $\Phi(x,y)= \frac{(x^2 y)^r }{\psi(x,y)}$, we have
	\[
	h(x,y,\mathbf{D}_0 (x,y))= \Phi(x,y) \mathbf{D}_{0}(x,y) \left(\rho (y)-x\right)^{-\frac{1}{2}}.
	\]
	Since $\psi(x,y)$ is analytic and nonzero and thus $\Phi(x,y)$ is analytic, the Taylor expansion of $\Phi(x,y)$ at $x=\rho (y)$ is given by
	\begin{equation}\label{Eq:Phi}
	\Phi(x,y)= a_0(y)+ a_1(y)\, \left(\rho (y)-x\right) \left(1+ o(1) \right),
	\end{equation}
	uniformly for  $y$ restricted to a neighborhood of $1$, where $a_0(y)$ and $a_1(y)$ are analytic for $y$ in a neighborhood of $1$ such that $a_0(1)\neq 0$.
	Combing eq.~(\ref{Eq:Phi}) and the uniform singular expansion~(\ref{Eq:0asy}) of $\mathbf{D}_0 (x,y)$ in Theorem~\ref{T:asyexp}, we derive
	\[
	h(x,y,\mathbf{D}_0 (x,y))= a_0(y)\pi(y) \left(\rho (y)-x\right)^{-\frac{1}{2}}  \left(1+ o(1) \right),
	\]
	uniformly for  $y$ restricted to a neighborhood of $1$.
	Setting $k(y)=a_0(y) \pi(y)$, we know $k(y)$  is analytic for $y$ in a neighborhood of $1$ and $k(1)=a_0(1) \pi(1)\neq 0$, whence  Claim 1 follows.
	
	Now we proceed by investigating the singular expansion of  $\mathbf{D}_g (x,y)$. Since $\mathbf{S}_{g}(x)$ is a polynomial of degree $6g-1$ and $h(x,y,\mathbf{D}_0 (x,y))$ has the uniform expansion~(\ref{Eq:h}),  we obtain
	\begin{equation*}
	\mathbf{S}_{g}(h(x,y,\mathbf{D}_0 (x,y)))= \kappa_{g}(3g-1)\, k(y)^{6g-1} \left(\rho (y)-x\right)^{-\frac{6g-1}{2}}  \left(1+ o(1) \right),
	\end{equation*}
	uniformly for  $y$ restricted to a neighborhood of $1$. Combing with eq.~(\ref{Eq:Dg2}) and the uniform singular expansion~(\ref{Eq:0asy}) of $\mathbf{D}_0 (x,y)$, we immediately derive
	\[
	\mathbf{D}_g (x,y)=\kappa_{g}(3g-1)\, \pi(y)k(y)^{6g-1} \, \left(\rho(y)-x \right)^{-\frac{6g-1}{2}} \left(1+ o(1) \right),
	\]
	uniformly for  $y$ restricted to a neighborhood of $1$.
	Set $\delta_g(y)=\kappa_{g}(3g-1) \pi(y)k(y)^{6g-1}$. It is clear that $\delta_g(y)$ is analytic and nonzero at $1$, whence  eq.~(\ref{Eq:gasy}) follows.
	
	A direct application of the standard transfer theorem to eq.~(\ref{Eq:gasy}) gives us the asymptotics~(\ref{Eq:coeffg}) of the coefficients $[x^{n} ] \mathbf{D}_g (x,y)$. The uniformity property follows from  the following estimation on the error term
	\[
	f(x,y):= o\left(\delta_g(y)\left(\rho(y)-x \right)^{-\frac{6g-1}{2}}    \right).
	\]	
	Based on the singular expansion of $\mathbf{D}_g (x,y)$, we can  assume that
	\[
	f(x,y)\leq K \delta_g(y)\left(\rho(y)-x \right)^{-(3g-1)},
	\]
	for some constant $K$. Since $\delta_g(y)$ is analytic at $1$, there exists a constant $\tilde{\delta}$ such that $|\delta_g(y)|<\tilde{\delta}$. Now we compute
	\begin{align*}
	|[x^{n} ] f(x,y)| &= \Big|\frac{1}{2  \pi\mathrm{i}} \int_{\Omega} f(x,y) \frac{\mathrm{d} x}{x^{n+1}}\Big|\\
	&\leq \frac{ K \tilde{\delta}}{2  \pi} \Big| \int_{\Omega} \left(\rho(y)-x \right)^{-(3g-1)} \frac{\mathrm{d} x}{x^{n+1}}\Big|\\
	&\leq \frac{ K \tilde{\delta}}{2  \pi} n^{(3g-2) } \big(\rho (y)
	\big)^{-n}\\
	&=o\Big(n^{\frac{6g-3}{2} } \big(\rho (y) \big)^{-n} \Big),
	\end{align*}
	where $\Omega$ is the Hankel contour around $\rho(y)$. The first equality follows from Cauchy's integral formula, and the third inequality is due to the estimation of the integral along each parts of the contour $\Omega$, the same as the proof of the  transfer theorem, see~\citet[pp. 390]{Flajolet:07a}, completing the proof.
\end{proof}
\begin{proof}[Sketch of the proof of Theorem~\ref{T:Dall}.]
	Using the approach of Theorem~\ref{T:Dg}, we prove the theorem via symbolic methods, considering a topological structure as the inflation of a shape $\omega$.
	As before, our construction utilizes simple combinatorial building blocks such as  stacks	$\mathcal{K}$, induced	stacks $\mathcal{N}$, stems $\mathcal{M}$, secondary structures $\mathcal{D}_{0}$, arcs $\mathcal{R}$, and unpaired vertices $\mathcal{X}$, where $\mathbf{X}(x)=x$ and $\mathbf{R}(x)=x^2$. Furthermore we let $y_1$, $y_2$, $y_3$, $y_4$, $y_5$ and $y_6$ denote the combinatorial markers for stacks, stems, hairpin loops, bulge loops, interior loops and multi-loops. Then
	\begin{align*}
	\mathcal{K}&=y_1\times \mathcal{R}^{r}\times\textsc{Seq}\left(\mathcal{R}\right),\\
	\mathcal{N}&=\mathcal{K}\times\Big( y_6 \times\big(\mathcal{D}_{0}^2 -1
	-2 \mathcal{X}\times\textsc{Seq}\left(\mathcal{X}\right)
	- \mathcal{X}^2\times\textsc{Seq}\left(\mathcal{X}\right)^2  \big) \\
	&\quad +y_4\times 2 \mathcal{X}\times\textsc{Seq}\left(\mathcal{X}\right)+ y_5\times\mathcal{X}^2\times\textsc{Seq}\left(\mathcal{X}\right)^2
	\Big),\\
	\mathcal{M}&=y_2\times\mathcal{K} \times \textsc{Seq}\left(\mathcal{N}\right),\\
	\mathcal{D}^{\omega}&=\mathcal{M}^k\times \mathcal{D}_{0}^{2 k+1},
	\end{align*}
	where the second equation marks whether the induced stack contains a bulge loop, an interior loop or a multi-loop. To get the information of the $i$-th parameters, we set $y_i=y$ and $y_j=1$ for $j\neq i$. Therefore we derive the generating function $\mathbf{D}^{\omega}$ and $\mathbf{D}_{g}^{i}(x,y)$.
\end{proof}
\begin{proof}[Proof of Theorem~\ref{T:SIrelation}]
	First we introduce some notions related to the decomposition of a diagram.
	Given a linear chord diagram $D$, an arc is called \emph{maximal} if it is maximal with respect to the partial order $ (i,j)\le (i',j') \quad \Longleftrightarrow \quad i'\le i\;\wedge j\le j'.$	
	Considering the left- and rightmost endpoints of a block containing
	some maximal arc induces a partition of the backbone into subsequent
	intervals. $D$ induces over each such interval a diagram, to which we
	refer to as a \emph{sub-diagram}.
	By construction, all maximal arcs of a fixed sub-diagram are contained in a unique
	block.
	
	In the  decomposition of a shape into a sequence of sub-diagrams,
	we distinguish the classes of sub-diagrams into two categories characterized by
	the unique arc-component containing all maximal arcs (maximal component). Namely,\\
	$\bullet$ sub-diagrams whose maximal arc-component contains only one arc, whose generating function is denoted by ${\bf T}^I_1(x,y,t)$,\\
	$\bullet$ sub-diagrams whose maximal arc-component is an (nonempty) irreducible diagram,  with generating function  ${\bf T}^I_2(x,y,t)$.\\
	
	{\bf Claim:} we have
	\begin{equation}\label{Eq:loopshape}
	\begin{aligned}
	{\bf S}^I(x,y,t)^{-1}&=1- {\bf T}^I_1(x,y,t)-{\bf T}^I_2(x,y,t)\\
	{\bf T}^I_1(x,y,t)&= x \left( {\bf T}^I_2(x,y,t)+ \frac{\left({\bf T}^I_1(x,y,t)+{\bf T}^I_2(x,y,t)\right)^2}{1-\left({\bf T}^I_1(x,y,t)+{\bf T}^I_2(x,y,t)\right)} \right)\\
	{\bf T}^I_2(x,y,t) &= 	{\bf S}^I(x,y,t)^{-1}{\bf I}^I\left(\frac{x{\bf S}^I(x,y,t)^2 }{1-x \left({\bf S}^I(x,y,t)^2-1\right)},y,t\right).
	\end{aligned}
	\end{equation}
	The first equation is implied by the decomposition of shapes into a sequence of
	two types of sub-diagrams.  Given a sub-diagram of the first type, the removal of the maximal arc-component (one arc)  generates again a sequence of
	two types of sub-diagrams. This sequence is neither empty nor  a sub-diagram of the first type, because a shape does not have $1$-arcs or parallel arcs. Thus we have the second equation.

	The third equation comes from the inflation of an irreducible shadow to sub-diagrams of the second type. Let $\zeta$ be a fixed irreducible shadow of genus $g$ having $n$ arcs and $l^I$ pseudoknots of $I$-type. Let $\mathcal{T}_2^{\zeta}$ denote the class of sub-diagrams of the second type, having $\zeta$ as the shadow of its unique maximal arc-component. Similarly we have  arcs $\mathcal{R}$, induced	stacks $\mathcal{N}$, stems $\mathcal{M}$, shapes $\mathcal{S}$, where $\mathbf{R}(x,y,t)=x$.  Then the inflation from  $\zeta$ to $\mathcal{T}_2^{\zeta}$ is described as follows
	\begin{align*}
	\mathcal{N}&=\mathcal{K}\times\left( \mathcal{S}^2-1\right) \\
	\mathcal{M}&=\mathcal{R} \times \textsc{Seq}\left(\mathcal{N}\right)\\
	\mathcal{T}_2^{\zeta}&=\mathcal{M}^n\times \mathcal{S}^{2 n-1}.
	\end{align*}
	Hence we derive the generating function
	\[
	{\bf T}^{\zeta}_2(x,y,t)=	{\bf S}^I(x,y,t)^{-1} \left(\frac{x{\bf S}^I(x,y,t)^2 }{1-x \left({\bf S}^I(x,y,t)^2-1\right)}\right)^n,
	\]
	which only depends on $n$. Summing over all irreducible shadows, gives rise to the third equation, completing the proof of the claim. Note that  genus and the number of  pseudoknots of type $I$ are both additive in this decomposition.
	
	Solving eqs.~(\ref{Eq:loopshape}) implies then eq.~(\ref{Eq:SIrelation}).
\end{proof}

\begin{proof}[Proofs of Theorem~\ref{T:loops} and Corollary~\ref{C:loops}]
	Let	$\mathbf{D}_g^{I} (x,y)$ denote the generating function of genus $g$ structures filtered by the number of $I$-type pseudoknots.
	Notice that a structure contains $I$-type	pseudoknots if and only if its corresponding shape contains also.  By the inflation method of Theorem~\ref{T:Dg}, we obtain
	\begin{equation}\label{Eq:loop}
	\mathbf{D}_g^{I} (x,y)= \mathbf{D}_{0}(x)\mathbf{S}_g^{I}\!\Big(\frac{x^{2r}  \mathbf{D}_{0}(x)^2}{1-x^2 -x^{2r} (\mathbf{D}_{0}(x)^2-1)},y\Big).
	\end{equation}
	Note that
	\begin{equation}\label{Eq:exp}
	\mathbb{E}(\mathbb{X}^I_{g,n})=
	\frac{[x^n]\partial_y \mathbf{D}_g^{I} (x,y)|_{y=1}}
	{[x^n] \mathbf{D}_g^{I} (x,1)}.
	\end{equation}
	By Corollary~\ref{C:shape2}, the polynomial $\partial_y \mathbf{S}_g^{H} (x,y)|_{y=1}$ has degree $6g-3$ in $x$. The proof of Theorem~\ref{T:gasy} shows that the subexponential factor of $[x^n]\partial_y \mathbf{D}_g^{H} (x,y)|_{y=1} $ is determined by the degree of  the polynomial $\partial_y \mathbf{S}_g^{H} (x,y)|_{y=1}$, whence we have
	\begin{align*}
	[x^n]\partial_y \mathbf{D}_g^{H} (x,y)|_{y=1} & = a_g n^{\frac{6g-5}{2}} \rho^{-n} \left(1 + O\left(\frac{1}{n}\right)\right),\\
	[x^n] \mathbf{D}_g^{H} (x,1) &= b_g n^{\frac{6g-3}{2}} \rho^{-n} \left(1 + O\left(\frac{1}{n}\right)\right),
	\end{align*}
	where $a_g$ and $b_g$ are some constants, $\rho$ is the dominant singularity of $\mathbf{D}_{0}(x)$.
	Therefore in view of eq.~(\ref{Eq:exp}) , we arrive at
	$
	\mathbb{E}(\mathbb{X}^H_{g,n})=O\left(n^{-1}\right).
	$
	Similarly we obtain $\mathbb{E}(\mathbb{X}^K_{g,n})=\mathbb{E}(\mathbb{X}^T_{g,n})=O(n^{-\frac{1}{2}})$ and $\mathbb{E}(\mathbb{X}^H_{g,n})=O\left(1\right)$.

	
	Now we set $g=\lambda=r=1$. Note that a structure of genus one has at most one $I$-type pseudoknots. Thus it suffices to compute the probability  of genus one structures containing such a pseudoknot, i.e., $\mathbb{P}(\mathbb{X}^I_{g,n}=1)$.
	We derive
	\begin{equation}\label{Eq:Htype}
	\mathbb{P}(\mathbb{X}^H_{1,n} =1)=\mathbb{E}(\mathbb{X}^H_{1,n})=
	\frac{[x^n]\partial_y \mathbf{D}_1^{H} (x,y)|_{y=1}}
	{[x^n] \mathbf{D}_1^{H} (x,1)}.
	\end{equation}
	Combining eq.~(\ref{Eq:loop}) and the singular expansion of $\mathbf{D}_{0}(x)$ in Theorem~\ref{T:asyexp}, we compute
	\begin{align*}
	\partial_y \mathbf{D}_1^{H} (x,y)|_{y=1}&=  \frac{1}{72} \left(\rho-x \right)^{-\frac{3}{2}} + o\big(\left(\rho-x \right)^{-\frac{3}{2}}\big), \\
	\mathbf{D}_1^{H} (x,1) &= \frac{1}{2592} \left(\rho-x \right)^{-\frac{5}{2}}-
	\frac{1}{256} \left(\rho-x \right)^{-\frac{3}{2}} + o\big(\left(\rho-x \right)^{-\frac{3}{2}}\big),
	\end{align*}
	where $\rho=\frac{1}{3}$.
	By Transfer Theorem (see Theorem 2 in Supplementary Material) 
	, we arrive at
	\begin{align*}
	[x^n]\partial_y \mathbf{D}_1^{H} (x,y)|_{y=1}= & \frac{\rho^{-\frac{3}{2}}}{72\, \Gamma(\frac{3}{2})}   \rho^{-n} n^{\frac{1}{2}} \left(1 + O\left(\frac{1}{n}\right)\right), \\
	[x^n]\mathbf{D}_1^{H} (x,1) =& \frac{\rho^{-\frac{5}{2}}}{2592\, \Gamma(\frac{5}{2})}  \rho^{-n} n^{\frac{3}{2}} \left(1+\frac{15}{8 n} +O\left(\frac{1}{n^2}\right) \right)\\
	&-  \frac{\rho^{-\frac{3}{2}}}{256\, \Gamma(\frac{3}{2})}   \rho^{-n} n^{\frac{1}{2}} \left(1 + O\left(\frac{1}{n}\right)\right).
	\end{align*}
	In view of eq.~(\ref{Eq:Htype}), we obtain the formula for $\mathbb{P}(\mathbb{X}^H_{1,n} =1)$ and proceed analogously for K, L, and M.
\end{proof}


\begin{center}
  {\bf ACKNOWLEDGMENTS} We wish to thank Christopher Barrett for stimulating discussions and the staff of the
  Biocomplexity Institute of Virginia Tech for their great support. Special thanks to Fenix W.D.~Huang for his
  help with the uniform sampler and for providing data on t-RNA structures.

\end{center}


\begin{center}
	{\bf AUTHOR DISCLOSURE STATEMENT}
\end{center}

The authors declare that no competing financial interests exist.


\bibliographystyle{spbasic}

\begin{thebibliography}{100}
	
	\bibitem[{Andersen \emph{et~al.}(2013)Andersen, Penner, Reidys, and
		Waterman}]{reidys:2013}
	Andersen J, Penner R, Reidys C, Waterman M (2013) Topological
	classification and enumeration of {RNA} structures by genus. J Math Biol
	67(5):1261--1278
	
	\bibitem[{Barrett \emph{et~al.}(2016)Barrett, Li, and Reidys}]{Barrett:16}
	Barrett C, Li T, Reidys C (2016) {RNA} secondary structures having a
	compatible sequence of certain nucleotide ratios. J Comput Biol Accepted
	
	\bibitem[{Berman \emph{et~al.}(1992)Berman, Olson, Beveridge, Westbrook, Gelbin,
		Demeny, Hsieh, Srinivasan, and Schneider}]{Berman:92}
	Berman H, Olson W, Beveridge D, Westbrook J, Gelbin A, Demeny T, Hsieh S,
	Srinivasan A, Schneider B (1992) The nucleic acid database. {A}
	comprehensive relational database of three-dimensional structures of nucleic
	acids. Biophysical Journal 63(3):751--759
	
	\bibitem[{Bon \emph{et~al.}(2008)Bon, Vernizzi, Orland, and Zee}]{Bon:08}
	Bon M, Vernizzi G, Orland H, Zee A (2008) Topological classification of {RNA}
	structures. J Mol Biol 379:900--911
	
	\bibitem[{Chen \emph{et~al.}(2000)Chen, Blasco, and Greider}]{Chen:00}
	Chen J, Blasco M, Greider C (2000) Secondary {Structure} of {Vertebrate}
	{Telomerase} {RNA}. Cell 100(5):503--514
	
	\bibitem[{Coimbatore~Narayanan \emph{et~al.}(2014)Coimbatore~Narayanan, Westbrook,
		Ghosh, Petrov, Sweeney, Zirbel, Leontis, and Berman}]{Narayanan:14}
	Coimbatore~Narayanan B, Westbrook J, Ghosh S, Petrov A, Sweeney B, Zirbel C,
	Leontis N, Berman H (2014) The {Nucleic} {Acid} {Database}: new features
	and capabilities. Nucleic Acids Research 42(Database issue):D114--D122
	
	\bibitem[{Euler(1752)}]{Euler:52}
	Euler L (1752) Elementa doctrinae solidorum. Novi Comment Acad Sc Imp Petropol
	4:109--160
	
	\bibitem[{Flajolet and Sedgewick(2009)}]{Flajolet:07a}
	Flajolet P, Sedgewick R (2009) Analytic Combinatorics. Cambridge University
	Press New York
	
	\bibitem[{Han \emph{et~al.}(2014)Han, Li, and Reidys}]{Han:14}
	Han H, Li T, Reidys C (2014) Combinatorics of $\gamma$-structures. J
	Comput Biol 21:591--608
	
	\bibitem[{Harer and Zagier(1986)}]{Harer:86}
	Harer J, Zagier D (1986) The {Euler} characteristic of the moduli space of
	curves. Invent Math 85:457--486
	
	\bibitem[{Hofacker \emph{et~al.}(1998)Hofacker, Schuster, and Stadler}]{Hofacker:98}
	Hofacker I, Schuster P, Stadler P (1998) Combinatorics of {RNA} secondary
	structures. Discrete Appl Math 88(1--3):207--237
	
	\bibitem[{Howell \emph{et~al.}(1980)Howell, Smith, and Waterman}]{Howell:80}
	Howell J, Smith T, Waterman M (1980) Computation of {Generating} {Functions}
	for {Biological} {Molecules}. SIAM J Appl Math 39(1):119--133
	
	\bibitem[{Huang and Reidys(2015)}]{Huang:14}
	Huang F, Reidys C (2015) Shapes of topological {RNA} structures.
	Mathematical Biosciences 270, Part A:57--65
	
	\bibitem[{Huang and Reidys(2016)}]{Huang:16}
	Huang F, Reidys C (2016) Topological language for {RNA}
	
	\bibitem[{Huang \emph{et~al.}(2013)Huang, Nebel, and Reidys}]{Huang:13}
	Huang F, Nebel M, Reidys C (2013) Generation of {RNA} pseudoknot structures
	with topological genus filtration. Mathematical Biosciences 245(2):216--225
	
	\bibitem[{Jin and Reidys(2008)}]{Jin-Reidys}
	Jin E, Reidys C (2008) Central and local limit theorems for {RNA} structures. J
	Theor Biol 250(3):547--559
	
	\bibitem[{Konings and Gutell(1995)}]{Konings:95}
	Konings D, Gutell R (1995) A comparison of thermodynamic foldings with
	comparatively derived structures of 16s and 16s-like {rRNAs}. RNA
	1(6):559--574
	
	\bibitem[{Li(2014)}]{Li-Phd}
	Li T (2014) Combinatorics of {Shapes}, {Topological} {RNA} {Structures} and
	{RNA}-{RNA} {Interactions}. Phd thesis, University of Southern Denmark,
	University of Southern Denmark
	
	\bibitem[{Li and Reidys(2011)}]{Li:11}
	Li T, Reidys C (2011) Combinatorial analysis of interacting {RNA} molecules.
	Math Biosci 233:47--58
	
	\bibitem[{Li and Reidys(2013)}]{Li:13}
	Li T, Reidys C (2013) The topological filtration of $\gamma$-structures.
	Math Biosc 241:24--33
	
	\bibitem[{Li and Reidys(2014)}]{Li:14}
	Li T, Reidys C (2014) A combinatorial interpretation of the
	$\kappa^{\star}_{g}(n)$ coefficients. arXiv:14063162v2
	\url{http://arxiv.org/abs/1406.3162}
	
	\bibitem[{Loebl and Moffatt(2008)}]{Loebl:08}
	Loebl M, Moffatt I (2008) The chromatic polynomial of fatgraphs and its
	categorification. Adv Math 217:1558--1587
	
	\bibitem[{Loria and Pan(1996)}]{Loria:96}
	Loria A, Pan T (1996) Domain structure of the ribozyme from eubacterial
	ribonuclease {P}. RNA 2(6):551--563
	
	\bibitem[{Massey(1967)}]{Massey:69}
	Massey W (1967) Algebraic Topology: An Introduction. Springer-Verlag, New York
	
	\bibitem[{Orland and Zee(2002)}]{Orland:02}
	Orland H, Zee A (2002) {RNA} folding and large $n$ matrix theory. Nuclear
	Physics B 620:456--476
	
	\bibitem[{Penner(2004)}]{Penner:03}
	Penner R (2004) Cell decomposition and compactification of {R}iemann's moduli
	space in decorated {T}eichm{\"u}ller theory. In: Tongring N, Penner R (eds)
	Woods Hole Mathematics-perspectives in math and physics, World Scientific,
	Singapore, pp 263--301
	
	\bibitem[{Penner and Waterman(1993)}]{Waterman:93}
	Penner R, Waterman M (1993) Spaces of {RNA} secondary structures. Adv Math
	217:31--49
	
	\bibitem[{Penner \emph{et~al.}(2010)Penner, Knudsen, Wiuf, and Andersen}]{penner:2010}
	Penner R, Knudsen M, Wiuf C, Andersen J (2010) Fatgraph models of proteins.
	Comm Pure Appl Math 63(10):1249--1297
	
	\bibitem[{Reidys(2011)}]{Reidys:11}
	Reidys C (2011) Combinatorial {Computational} {Biology} of {RNA}. Springer New
	York, New York, NY
	
	\bibitem[{Reidys \emph{et~al.}(2010)Reidys, Wang, and Zhao}]{Reidys:10}
	Reidys C, Wang R, Zhao A (2010) Modular, \$k\$-{Noncrossing} {Diagrams}.
	The Electronic Journal of Combinatorics 17(1):R76
	
	\bibitem[{Reidys \emph{et~al.}(2011)Reidys, Huang, Andersen, Penner, Stadler, and
		Nebel}]{Huang:11}
	Reidys C, Huang F, Andersen J, Penner R, Stadler P, Nebel M (2011)
	Topology and prediction of {RNA} pseudoknots. Bioinformatics pp 1076--1085
	
	\bibitem[{Schmitt and Waterman(1994)}]{Waterman:94a}
	Schmitt W, Waterman M (1994) Linear trees and \textsc{RNA} secondary structure.
	Disc Appl Math 51:317--323
	
	\bibitem[{Smith and Waterman(1978)}]{Waterman:78aa}
	Smith T, Waterman M (1978) {RNA} secondary structure. Math Biol 42:31--49
	
	\bibitem[{Staple and Butcher(2005)}]{Staple:05}
	Staple D, Butcher S (2005) Pseudoknots: {RNA} {Structures} with {Diverse}
	{Functions}. PLOS Biol 3(6):e213
	
	\bibitem[{Stein and Waterman(1979)}]{Stein:79}
	Stein P, Waterman M (1979) On some new sequences generalizing the {Catalan} and
	{Motzkin} numbers. Discrete Math 26(3):261--272
	
	\bibitem[{Tsukiji \emph{et~al.}(2003)Tsukiji, Pattnaik, and Suga}]{Tsukiji:03}
	Tsukiji S, Pattnaik S, Suga H (2003) An alcohol dehydrogenase ribozyme. Nature
	Structural Biology 10(9):713--717
	
	\bibitem[{Tuerk \emph{et~al.}(1992)Tuerk, MacDougal, and Gold}]{Tuerk:92}
	Tuerk C, MacDougal S, Gold L (1992) {RNA} pseudoknots that inhibit human
	immunodeficiency virus type 1 reverse transcriptase. Proceedings of the
	National Academy of Sciences of the United States of America
	89(15):6988--6992
	
	\bibitem[{Vernizzi \emph{et~al.}(2005)Vernizzi, Orland, and Zee}]{Vernizzi:05}
	Vernizzi G, Orland H, Zee A (2005) Enumeration of {RNA} {Structures} by
	{Matrix} {Models}. Physical Review Letters 94(16):168,103
	
	\bibitem[{Waterman(1978)}]{Waterman:78s}
	Waterman M (1978) Secondary structure of single-stranded nucleic acids. In:
	Rota GC (ed) Studies on foundations and combinatorics, Advances in
	mathematics supplementary studies, Academic Press N.Y., vol~1, pp 167--212
	
	\bibitem[{Waterman(1979)}]{Waterman:79a}
	Waterman M (1979) Combinatorics of {RNA} {Hairpins} and {Cloverleaves}. Stud
	Appl Math 60(2):91--98
	
	\bibitem[{Westhof and Jaeger(1992)}]{Westhof:92}
	Westhof E, Jaeger L (1992) {RNA} pseudoknots. Curr Opin Chem Biol 2:327--333
	
	\bibitem[{Whitney(1932)}]{Whitney:32}
	Whitney H (1932) Congruent {Graphs} and the {Connectivity} of {Graphs}.
	American Journal of Mathematics 54(1):150--168
	
\end{thebibliography}


\end{document}